\documentclass[11pt, a4paper, oneside, reqno]{amsart}

\newcommand{\X}{\mathcal{X}}
\newcommand{\A}{\mathcal{A}}
\newcommand{\B}{\mathcal{B}}

\newcommand{\Real}{\mathbb{R}}
\newcommand{\Regret}{\mathbf{Reg}^s}
\newcommand{\DRegret}{\mathbf{Reg}^d}

\newcommand{\norm}[1]{\|#1\|}
\newcommand{\inner}[2]{\langle #1,#2 \rangle}

\newtheorem{theorem}{Theorem}[section]
\newtheorem{definition}[theorem]{Definition}
\newtheorem{lemma}[theorem]{Lemma}
\newtheorem{assumption}[theorem]{Assumption}
\newtheorem{remark}[theorem]{Remark}

\usepackage{multirow}
\usepackage{makecell}
\usepackage{arydshln}

\makeatletter
\def\@settitle{\begin{center}%
		\baselineskip14\p@\relax
		\normalfont\LARGE\scshape\bfseries
		\@title
	\end{center}%
}
\makeatother

\makeatletter

\def\subsection{\@startsection{subsection}{2}%
	\z@{.5\linespacing\@plus.7\linespacing}{.5\linespacing}%
	{\normalfont\large\bfseries}}

\def\subsubsection{\@startsection{subsubsection}{3}%
	\z@{.5\linespacing\@plus.7\linespacing}{.5\linespacing}%
	{\normalfont\itshape}}

\usepackage[colorlinks	= true,
			raiselinks	= true,
			linkcolor	= MidnightBlue,
			citecolor	= Mahogany,
			urlcolor	= ForestGreen,
			pdfauthor	= {Peyman Mohajerin Esfahani},
			pdftitle	= {},
			pdfkeywords	= {},   
			pdfsubject	= {},
			plainpages	= false]{hyperref}
\usepackage[usenames, dvipsnames]{color}
\usepackage{dsfont,amssymb,amsmath,subfig, graphicx,fancyhdr,mdframed}
\usepackage{amsfonts,dsfont,mathtools, mathrsfs,amsthm,wrapfig} 
\usepackage[]{siunitx}
\usepackage{ragged2e}
\usepackage{enumitem}

\allowdisplaybreaks
\usepackage{algorithm}
\usepackage[noend]{algpseudocode}

\allowdisplaybreaks
\date{\today}

\patchcmd\@setauthors
  {\MakeUppercase{\authors}}
  {\authors}
  {}{}

\graphicspath{{images/}}

\title[Adaptive Composite Online Optimization]{Adaptive Composite Online Optimization: \\ Predictions in Static and Dynamic Environments}

\author[P. Zattoni Scroccaro]{Pedro Zattoni Scroccaro}
\author[A. S. Kolarijani]{Arman Sharifi Kolarijani}
\author[P. {Mohajerin Esfahani}]{Peyman {Mohajerin Esfahani}}

\thanks{The authors are with the Delft Center for Systems and Control, TU Delft, The Netherlands. (e-mails: P.ZattoniScroccaro@tudelft.nl; A.SharifiKolarijani@tudelft.nl; P.MohajerinEsfahani@tudelft.nl.). This research is partially supported by the ERC grant TRUST-949796.}

\begin{document}

\maketitle
	
\begin{abstract}   
In the past few years, Online Convex Optimization (OCO) has received notable attention in the control literature thanks to its flexible real-time nature and powerful performance guarantees. In this paper, we propose new step-size rules and OCO algorithms that simultaneously exploit {\em gradient predictions}, {\em function predictions} and {\em dynamics}, features particularly pertinent to control applications. The proposed algorithms enjoy static and dynamic regret bounds in terms of the dynamics of the reference action sequence, gradient prediction error, and function prediction error, which are generalizations of known regularity measures from the literature. We present results for both convex and strongly convex costs. We validate the performance of the proposed algorithms in a trajectory tracking case study, as well as portfolio optimization using real-world datasets.
\end{abstract}

\section{Introduction}\label{sec:introduction}

The standard framework of Online Convex Optimization (OCO) can be described as a game between a Player and Nature, played over $T$ rounds.
Let $\mathcal{A}$ be the Player's action space.
Suppose that $\X\subseteq \mathcal{A}$ is a convex set representing the set of possible actions of the Player.
Moreover, let $\mathcal{F}$ denote a set of convex functions available to Nature.
At each round $t$, the Player chooses an action $x_t \in \X$.
After the Player commits with an action, Nature reveals a convex cost $f_t : \X \to \Real$ where $f_t\in\mathcal{F}$.
The Player suffers the loss $f_t(x_t)$.
The goal of the Player is to perform as well as possible against the costs chosen by Nature. (See \cite{hazan2016introduction}, \cite{cesa2006prediction} and \cite{shalev2012online} for in-depth studies of fundamental theories of OCO and its many applications.)

A common metric to evaluate the performance of the Player is the so-called \textit{static regret} defined as
\begin{align} \label{eq:regret}
    \Regret_T := \sum^T_{t=1} f_t(x_t) - \min_{x \in \X} \sum^T_{t=1} f_t(x).
\end{align}
Intuitively, this metric quantifies how well the Player performs against the \emph{best fixed} action computed in hindsight. Based on this notion of regret, OCO algorithms are designed such that the resulting action sequence $\{x_t\}^T_{t=1}$ guarantees a \textit{sub-linear} regret w.r.t. $T$, i.e., $\lim_{T\to\infty}(\Regret_T/T)=0$. In other words, such OCO strategies perform (on average) as well as the best fixed action in hindsight. A standard algorithm to choose $x_t$ is called Online Mirror Descent (OMD) algorithm
\begin{align} \label{alg:OMD}
\tag{OMD}
x_{t+1} = \text{arg}\min_{x\in\X} \ \big\{ \eta_t \inner{\nabla f_t(x_t)}{x}  + \B_h(x,x_t)\big\},
\end{align}
where $\eta_t$ denotes the step-size and $\B_h$ is the Bregman divergence functional \cite{cesa2006prediction}. By choosing $\eta_t$ appropriately, Algorithm \ref{alg:OMD} guarantees ${\Regret_T \leq O(\sqrt{T})}$ \cite{bubeck2011introduction} or ${\Regret_T \leq O(\log(T))}$ \cite{shalev2007logarithmic}, based on the regularity of the cost set $\mathcal{F}$. Moreover, Abernethy et al. \cite{abernethy2008optimal} showed that these regret rates are in fact optimal by the \emph{minimax} formulation of OCO problems. 

However, there are many OCO problems in which the Player and Nature \emph{do not exactly follow} the rules of the sequential game mentioned above. In this paper, we focus on the case of \emph{OCO with predictions}. In these scenarios, we assume access to \emph{predictions} about the costs of the problem being studied, and we use OCO algorithms combined with these predictions in order to achieve improved regret guarantees. For instance, if our OCO problem is related to estimating the evolution of dynamical parameters of a system, predictions could come from a dynamical model we have of the system (see Section \ref{sec:numerical}). This approach is inspired by the classical control theory literature, in which dynamical and/or predictive models of the system being controlled are almost always assumed to exist. Moreover, there has been recent interest from both the online learning and controls communities in combining Online Convex Optimization techniques to control problems, e.g., \cite{pmlr-v97-agarwal19c, agarwal2019logarithmic, pmlr-v117-hazan20a, wagener2019online}. Also, most of the results presented in this work apply to problems with composite costs with nonsmooth components (e.g., $\|\cdot\|_1$). These results open up even more possibilities of connections with control applications, for instance, $\ell_1$ optimization for sparse networked feedback control \cite{nagahara2013sparse, nagahara2015maximum}. Therefore, we hope that this work lays a theoretical foundation and also inspires new works in the intersection of OCO and control theory.

Next, we formally define important notions that will be used throughout the paper.

\subsection{Gradient Predictions}

The minimax regret bounds for OCO algorithms are derived assuming a worst-case (i.e. fully adversarial) cost sequence $\{f_t\}^T_{t=1}$. The cost sequence is however \emph{not} completely adversarial in many practical OCO problems \cite{rakhlin2013online}. In such problems, the Player can (partially) \emph{predict} the \emph{unseen} cost $f_t$ at round $t$, before deciding its action $x_t$. \footnote{This assumption deviates from the standard OCO protocol, where Nature reveals $f_t$ only \emph{after} the Player chooses $x_t$.} It is hence natural to expect that one can possibly exploit the predictability of an OCO problem to achieve tighter regret bounds.

A generic notion of the predictability of Nature's moves can be stated as follows \cite{rakhlin2013online}. At the outset of each round $t \in [T]$, the Player has access to the value of a function
 \begin{align*}
    M_t: \X^{t-1} \times \mathcal{F}^{t-1} \times \mathcal{I}^{t-1} \to \mathcal{P},     
 \end{align*}
where $\mathcal{I}$ denotes some information space provided to the Player via an exogenous source and $\mathcal{P}$ is the space to which each predictable entity belongs.  In particular, a certain class of OCO problems with predictability is the class of OCO problems with \emph{gradient predictions}. Observe that here $\mathcal{P}\subseteq\mathcal{A}^*$, where $\mathcal{A}^*$ is the dual space of the action space $\mathcal{A}$. To exploit gradient predictions in OCO problems, Rakhlin and Sridharan \cite{rakhlin2013online} proposed the Optimistic Mirror Descent (OptMD) algorithm
\begin{align} \label{alg:OptMD}
\tag{OptMD}
\begin{aligned}
x_t &= \text{arg}\min_{x\in\X} \ \big\{ \eta_t \inner{M_t}{x}  + \B_h(x,y_{t-1})\big\} \\
y_t &= \text{arg}\min_{y\in\X} \ \big\{ \eta_t \inner{\nabla f_t(x_t)}{y}  + \B_h(y,y_{t-1})\big\},
\end{aligned}
\end{align}
where $\{M_t\}_{t=1}^T$ is a generic gradient prediction sequence.\footnote{Notice that Algorithm \ref{alg:OptMD} reduces to Algorithm \ref{alg:OMD} when $M_t=0$.}
In \cite{rakhlin2013optimization}, the authors further provided an adaptive step-size rule for Algorithm \ref{alg:OptMD} such that $\Regret_T \leq O(1+\sqrt{D_T})$, where
\begin{align} \label{eq:D_T}
    D_T := \sum^T_{t=1} \norm{\nabla f_t(x_t) - M_t}^2_*.
\end{align}
When the Player has access to $\nabla f_t(\cdot)$ \emph{before} choosing $x_t$, we say that the Player has access to \emph{perfect gradient predictions}. In this scenario, Ho-Nguyen and Kılınc-Karzan \cite{ho2019exploiting} showed that by setting $M_t:=\nabla f_t(y_{t-1})$, $\eta_t \leq 1/\beta$ and when $\mathcal{F}$ represents $\beta$-\textit{smooth} functions, Algorithm \ref{alg:OptMD} guarantees $\Regret_T \leq O(1)$.

\subsection{The Problem with $D_T$}
\label{subsec:DT_problem}

In the following, we argue that regret bounds given in terms of $D_T$ are not suitable for exploiting gradient predictions, mainly because $x_t$ \textbf{is not available} at the beginning of round $t$ (see Algorithm \ref{alg:OptMD}). In some works that prove regret bounds in terms of $D_T$ (e.g. \cite{rakhlin2013optimization,jadbabaie2015online}) it is argued that for ``predictable sequences'', external knowledge of the gradient sequence can be used to achieve tighter regret bounds. For example, in \cite{jadbabaie2015online}, the authors state that: ``...one can get a tighter bound for regret once the learner advances a sequence of conjectures $\{M_t\}_{t=1}^T$ well-aligned with the gradients''. However, consider the following scenario: at the beginning of round $t$, the Player has access to a prediction of $\nabla f_t(\cdot)$, namely $\nabla \hat{f}_t(\cdot)$. Now, based on those regret bounds given in terms of $D_T$, how one would choose $M_t$ when using Algorithm \ref{alg:OptMD}? Naturally, we want to choose $M_t$ so that $D_T$ is as small as possible (recall that $D_T = \sum_{t=1}^T \|\nabla f_t(x_t) - M_t\|^2_*$). However, since $x_t$ is not available at the beginning of round $t$, we cannot set $M_t = \nabla \hat{f}_t(x_t)$. Thus, from these regret bounds, it is not clear how one should choose $M_t$ in order to exploit this type of gradient prediction. Moreover, Ho-Nguyen and Kılınc-Karzan \cite{ho2019exploiting} showed that when perfect gradient predictions are available (that is, $\nabla \hat{f}_t(\cdot) = \nabla f_t(\cdot)$), constant static regret is achievable. Still, this constant regret result is not recovered by the regret bound $\Regret_T \leq O(1+\sqrt{D_T})$ given in \cite{rakhlin2013optimization}, even when perfect gradient predictions are available. In fact, since smoothness of the cost is not assumed in \cite{rakhlin2013optimization}, if it was possible to choose $M_t$ such that $D_T = 0$ (i.e., such that $\Regret_T \leq O(1)$), this would contradict the lower bound for first-order optimization methods \cite{nesterov2004introductory}, \cite[Remark 1]{yang2014regret}. Therefore, we conclude that in order to effectively exploit gradient predictions, a different approach must be used.

\subsection{Dynamic Environments and Regularity Measures}

In the regret notion \eqref{eq:regret}, the Player's cumulative loss competes against the loss of the best fixed action in hindsight. There are, on the other hand, many OCO problems where the best fixed action is not accessible or does not exist \cite{besbes2015non}. Thus, in those cases, the use of the regret \eqref{eq:regret} is not convenient anymore. The term OCO problems in \emph{dynamic environments} is used in the literature for such problems \cite{hall2015online}. 

To generalize the standard regret notion in order to tackle these scenarios, Zinkevich \cite{zinkevich2003online} proposed to compare the Player's performance against a general dynamical \textit{reference sequence} $\{ u_t \}^T_{t=1}\in\X^T$.
The resulting metric is called the \textit{dynamic regret}, defined as
\begin{equation} \label{eq:dynamic_regret}
    \DRegret_T := \sum^T_{t=1} f_t(x_t) - \sum^T_{t=1} f_t(u_t).
\end{equation}
Unfortunately, it is impossible to achieve a sub-linear dynamic regret for an arbitrarily chosen $\{ u_t \}^T_{t=1}$ \cite{mokhtari2016online}. Thus, in order to achieve \emph{meaningful} dynamic regret bounds, it is common to place extra \emph{regularity assumptions} on the costs and/or the reference sequence. For example, Hall and Willett \cite{hall2013dynamical} consider the bounded variability of the reference sequence in terms of
\begin{equation} \label{eq:C_T}
    C_T := \sum^T_{t=1} \norm{u_{t+1} - u_t}.
\end{equation}
For convex costs, the authors show that Algorithm \ref{alg:OMD} guarantees ${\DRegret_T \leq O(\sqrt{T}(1+C_T))}$. The authors further consider that the Player has access to dynamical models ${\Phi_t: \X \to \X}$ of the reference sequence, that is, models that approximate the true dynamical models $\Phi^\star_t$, i.e., $u_{t+1} = \Phi^\star_t(u_t)$.
They employ $\Phi_t(x_t)$ instead of $x_t$ in Algorithm \ref{alg:OMD}
and prove $\DRegret_T \leq O(\sqrt{T}(1+C'_T))$, where
\begin{align} \label{eq:C'_T}
    C'_T := \sum^T_{t=1} \norm{u_{t+1} - \Phi_t(u_t)}.
\end{align}
When $\Phi_t$ approximates the true dynamics well enough, we may have $C'_T\leq C_T$, which in turn implies tighter dynamic regret bounds. Subsequently, Jadbabaie et al. \cite{jadbabaie2015online} studied dynamical environments to account for the cases with gradient predictions. The authors show that Algorithm \ref{alg:OptMD} guarantees $\DRegret_T \leq O\left(\sqrt{1+D_T}(1+C_T)\right)$ in such cases. Finally, using an expert-based algorithm called \emph{Ader}, Zhang et al. \cite{zhang2018adaptive} showed that it guarantees the optimal bound $\DRegret_T \leq O\left(\sqrt{T(1+C'_T)}\right)$.

Another important regularity measure popular in the literature is the \textit{temporal variability} of the cost sequence
\begin{equation}
    \label{eq:V_T}
    V_T := \sum_{t=2}^T \max_{x \in \X} |f_t(x) - f_{t-1}(x)|.
\end{equation}
In the setting of stochastic optimization with noisy gradients, Besbes et al. \cite{besbes2015non} show that a restarted gradient descent algorithm incurs dynamic regret bounded by $O(T^{2/3}(1 + v)^{1/3})$, where $v$ is an upper bound of $V_T$ known in advance, and Jadbabaie et al. \cite{jadbabaie2015online} provided an algorithm which guarantees a dynamic regret bound of $\tilde{O}(\sqrt{D_T + 1} + \min\{\sqrt{(D_T + 1)C_T}, (D_T+1)^{1/3}T^{1/3}V_T^{1/3}\})$ \footnote{The $\tilde{O}$ notation hides poly-logarithmic terms.}, for the specific case when the regret is defined w.r.t. the reference sequence $u_t = \arg\min_{x\in\X} f_t(x)$, also known as \textit{restricted dynamic regret} \cite{campolongo2021closer}.

\subsection{Composite Cost, Implicit Updates and Function Predictions}

A cost function $f_t$ is called composite if it can be decomposed as $f_t(\cdot) = s_t(\cdot) + r_t(\cdot)$. For example, Duchi et al. \cite{duchi2010composite} consider the case when $r_t(\cdot) = r(\cdot)$ for all $t$, and proposes the Composite Objective Mirror Descent (COMID) algorithm
\begin{align} \label{alg:COMID}
\tag{COMID}
x_{t+1} = \text{arg}\min_{x\in\X} \ \big\{ \eta_t \inner{\nabla s_t(x_t)}{x}  + \eta_t r(x) +  \B_h(x,x_t)\big\},
\end{align}
where differently from \ref{alg:OMD}, the fixed part $r(\cdot)$ is not linearized. This can be advantageous when, for example, $r(\cdot) = \|\cdot\|_1$. In this case, using \ref{alg:COMID} would lead to sparse updates, whereas \ref{alg:OMD} would not \cite{duchi2010composite}. In the offline optimization literature (i.e., $f_t(\cdot) = f(\cdot)$ for all $t$), algorithms that partially linearize the cost function are called \textit{proximal gradient methods} \cite{parikh2014proximal, beck2017first}. These algorithms are usually used when $s$ is smooth, but $r$ is not. Then, by linearizing only the smooth component of $f$, a proximal gradient method can lead to convergence rates that match the one of \ref{alg:OMD} for smooth costs (e.g. $O(1/T)$ rate instead of $O(1/\sqrt{T})$). Intuitively, when smoothness is necessary to prove a convergence rate for some first-order algorithms, one can usually deal with nonsmooth components by not linearizing them in the proximal updates.

Somewhat related to proximal gradient methods are the so-called implicit updates, also known as Implicit Online Mirror Descent (IOMD) \cite{kivinen1997exponentiated, kulis2010implicit, campolongo2020temporal}
\begin{align} \label{alg:IOMD}
\tag{IOMD}
x_{t+1} = \text{arg}\min_{x\in\X} \ \big\{ \eta_t f_t(x) +  \B_h(x,x_t)\big\}.
\end{align}
Kulis and Bartlett \cite{kulis2010implicit} proved regret bounds for \ref{alg:IOMD} that match the ones from \ref{alg:OMD}. McMahan \cite{mcmahan2010unified} and Song et al. \cite{song2018fully} quantify the advantage of implicit updates through non-negative, data-dependent quantities. Recently, Campolongo and Orabona \cite{campolongo2020temporal} showed that an adaptive version of \ref{alg:IOMD} guarantees $O(\min\{V_T,\sqrt{T}\})$. Moreover, in dynamic environments, Campolongo and Orabona \cite{campolongo2021closer} show that a version of \ref{alg:IOMD} guarantees $O(\min\{V_T,\sqrt{T(1 + \tau)}\})$, where $\tau$ is a known upper bound of $C_T$. When $\tau$ is not known, a similar bound $\tilde{O}(\min\{V_T,\sqrt{T(1 + C_T)}\})$ can be achieved by combining implicit updates with experts and strongly-adaptive algorithms \cite{campolongo2021closer}.

When a linearized version of the cost $f_t$ is used in our OCO strategy, e.g., \ref{alg:OMD} algorithm, it is natural to expect that we only need \textit{gradient predictions} to exploit information of unseen costs, as it is done in the \ref{alg:OptMD} algorithm. However, when using strategies that partially linearize the cost $f_t$ (or do not linearize it at all), one should not hope that gradient predictions of the cost can be effectively used. Therefore, in order to exploit predictive information about cost functions, we will require \textit{gradient predictions} of its linearized component and \textit{function predictions} of its non-linearized component. For example, for the composite cost $f_t(x) = s_t(x) + r_t(x)$, if we decide to linearize $s_t(x)$ and not linearize $r_t(x)$, we will require gradient predictions of $s_t$ and function predictions of $r_t$, denoted as $\hat{r}_t$.

\subsection{Problem Description and Related Works}

In this paper, we consider OCO problems with composite costs of the form
$$
f_t(\cdot) = s_t(\cdot) + r_t(\cdot),
$$
in both static and dynamic environments. Recall that Ho-Nguyen and Kılınç-Karzan \cite{ho2019exploiting} observed that perfect gradient predictability in the form of $M_t=\nabla f_t(y_{t-1})$ implies that Algorithm \ref{alg:OptMD} guarantees constant static regret.
Motivated by this observation and the discussion presented in Subsection \ref{subsec:DT_problem}, we extend this idea to the case of an ``{imperfect}'' gradient predictability.
To do so, we introduce the \textit{gradient prediction error measure}
\begin{align} \label{eq:D_T'}
    D_t' := \sum^{t}_{\tau=1} \norm{\nabla s_\tau(y_{\tau-1}) - \nabla \hat{s}_\tau(y_{\tau-1})}^2_*,
\end{align}
where $\{y_{\tau-1}\}_{\tau=1}^t$ are points generated by an online algorithm. Notice that we changed the notation from $M_t$ to $\nabla \hat{s}_\tau(y_{\tau-1})$. We do it so that the connection between the gradient of $s_t$ and the gradient predictions is clearer. Also, notice that this measure refers to gradient predictions only for the $s_t$ component of $f_t$. Thus, we also introduce the \textit{function prediction error measure}
\begin{align} \label{eq:V_T'}
    V_t' := \sum_{\tau=1}^t | r_\tau(x_\tau) - \hat{r}_\tau(x_\tau) + \hat{r}_\tau(y_\tau) - r_\tau(y_\tau) |,
\end{align}
where $\{x_\tau\}_{\tau=1}^t$ and $\{y_\tau\}_{\tau=1}^t$ are points generated by an online algorithm. When $s_t = 0$, $V_T'$ can be interpreted as a generalization of $V_T$ for the case when function predictions are available. Namely, when function predictions are not available, by setting $\hat{r}_t  = r_{t-1}$, we get $V'_T \leq 2 V_T$. Moreover, in dynamic environments, we further suppose that the Player has access to a (possibly approximate) dynamical model $\Phi_t$ of the reference sequence $\{u_t\}_{t=1}^T$. This is a useful assumption, which has been used in practical applications of OCO algorithms \cite{shahrampour2017distributed, wagener2019online}. We are now set to state the problem considered in this paper.

\noindent\textbf{Problem:} Design and analyze OCO algorithms such that the corresponding regret bounds exploit
\begin{itemize}
    \item (possibly imperfect) gradient and/or function predictions of the components of the cost sequence $\{f_t\}_{t=1}^T$;
    \item (possibly approximated) dynamical models $\Phi_t$ of the reference sequence $\{u_t\}_{t=1}^T$.
\end{itemize}

Other than the works already mentioned in Section \ref{sec:introduction}, several studies in the literature propose algorithms that take advantage of the predictability of the cost sequence. Several works exploit predictions in OCO problems with switching costs. In this scenario, at round $t$, the Player suffers the loss $f_t(x_t,x_{t-1}) = c_t(x_t) + \gamma \norm{x_t - x_{t-1}}$, where $c_t$ is a convex function and $\gamma \norm{x_t - x_{t-1}}$ is the switching cost. In order to exploit predictions in these problems, it is usually necessary to have a \textit{window of future cost predictions} \cite{chen2015online, chen2016using, lin2020online, li2020leveraging, li2020online}. Another application where predictions have been used is the so-called \textit{online control} problem. For this class of problems, due to the dynamics of the system, the cost $f_t$ may depend on the whole history of previous actions, and a window of predictions is again necessary \cite{li2019online, yu2020power, li2022robustness}. Thus, since in this work, the cost at time $t$ only depends on $x_t$ and we only use predictions about the very next cost, results on OCO with switching costs and online control are not directly comparable to this paper's results. Dekel et al. \cite{dekel2017online} study Online Linear Optimization.
The authors suppose that at the outset of each round, the Player has access to a vector (or hint) that is correlated with the cost to be incurred to the Player. If \textit{all} hints are sufficiently good and the action set possesses certain geometrical properties, the authors show $\Regret_T \leq O(\log(T))$. Recently, Bhaskara et al. \cite{bhaskara2020online} extended this result to the case when not all hints are correlated with the true cost vector. In dynamic environments, Lesage-Landry et al. \cite{lesage2020predictive} showed that tighter dynamic regret bounds can be achieved by only using predictions that meet certain conditions. In \cite{ravier2019prediction}, the authors employ gradient predictions in order to obtain possibly tighter dynamic regret bounds. However, the proposed approach yields regret bounds that lack worst-case guarantees.  In \cite{chang2021online}, the authors propose an online optimistic Newton method that exploits gradient and hessian predictions and prove dynamic regret bounds for this algorithm.

\subsection{Contributions and Organization}

A summary of the main results is now given.
\begin{enumerate}[label=(\roman*)]
    \item \textbf{Static regret for convex costs}: In static environments, we propose a novel algorithm that uses step-sizes that adapt to the quality of gradient and function predictions, guaranteeing $\Regret_T \leq O\left(1 + \sqrt{D_T'} + \min\left\{ V_T', \sqrt{T} \right\}\right)$ for convex costs (Theorem \ref{theorem:OptCMD_convex}). This result generalizes the best-case $\Regret_T \leq O(1)$ \cite{ho2019exploiting} and worst-case $\Regret_T \leq O(\sqrt{T})$ \cite{zinkevich2003online} regret rates. 
    
    \item \textbf{Static regret for strongly convex costs}: When the costs are strongly convex and we have access to the $r_t$ components, we propose an adaptive step-size $\eta_t$  which improves the regret to $\Regret_T \leq O(1 + \log(1 + D'_T))$ (Theorem \ref{theorem:OptCMD_strong}).  This result generalizes the best-case $\Regret_T \leq O(1)$ \cite{ho2019exploiting} and worst-case $\Regret_T \leq O(\log (T)$ \cite{hazan2007logarithmic} regret rates. 

    \item \textbf{Dynamic regret for convex costs}: For dynamic environments, we introduce a new variant of Algorithm \ref{alg:OptMD} that simultaneously exploits gradient predictions, function predictions, and the dynamics of the reference sequence. We show that it guarantees the dynamic regret bound
    $\DRegret_T \leq O\bigg((1 + C'_T )\Big(1 + \sqrt{D_T'} + \min\left\{ V_T', \sqrt{(1 + C'_T) T} \right\}\Big)\bigg)$ 
    for convex costs (Theorem \ref{theorem:OptDCMD_convex}).
    
    \item \textbf{Dynamic regret for implicit updates}: Using fully implicit updates, we show that the proposed algorithm guarantees the dynamic regret bound $ \DRegret_T \leq O \left( \min \left\{ V_T', \ \sqrt{(1 + \tau)T} \right\} \right)$, where $\tau$ is a known upper bound to $C_T'$ (Theorem \ref{theorem:OptDCMD_implicit}). This result generalizes the dynamic regret bounds of \cite{campolongo2021closer}, for the case of function predictions.
    
    \item \textbf{Dynamic regret for fully adaptive step-size}: Finally, when we have access to the $r_t$ component of the costs, we propose a step-size $\eta_t$ which adapts to gradient predictions and $C_t'$ on the fly. The resulting algorithm guarantees $\DRegret_T \leq O\left(\sqrt{(\theta_T + D'_T)(1 + C'_T)}\right)$, where $\theta_T$ is a parameter used to control the size of the step-size. (Theorem \ref{theorem:OptDMD_convex_2}).
\end{enumerate}
For the ease of the readers, we also provide Tables~\ref{table:comparison}-\ref{table:func_pred} in Appendix~\ref{app:tables} presenting the above contributions within the existing OCO literature reviewed earlier, with a particular focus on the predictions and composite features in the context of static regret bounds. 

The organization of the paper is as follows. The main results of this study are provided in Section \ref{sec:main_results}. To improve the flow of the paper, we moved the proofs of our main results to Section \ref{sec:proofs}. Numerical experiments are presented in Section \ref{sec:numerical}. Finally, in Appendix \ref{app:tables}, we present tables that position our work with respect to a body of the OCO literature.

\section{Main Results}\label{sec:main_results}

We start with some definitions and assumptions that will be used throughout the paper.

\subsection{Mathematical Preliminaries}

Let the action set $\X \subset \Real^n$.
We denote by $\norm{\cdot}_*$ the dual norm of $\norm{\cdot}$. Also, we define $[T] := \{1,2,\dots,T\}$. 

\begin{definition}[Bregman divergence]\label{def:bregman}
Let $h: \X \to \Real$ be a differentiable convex function. The Bregman divergence of $x,y \in \X$, w.r.t. the function $h$ is $\B_h(x,y) := h(x) - h(y) - \inner{\nabla h(y)}{x-y}$.
\end{definition}

\begin{definition}[$\alpha$-Strong convexity]\label{def:strong_convexity}
A function $f: \X \to \Real$ is $\alpha$-\emph{strongly convex} w.r.t. a norm $\norm{\cdot}$ if $f(x) - f(y) \leq \inner{\nabla f(x)}{x-y} - \frac{\alpha}{2} \norm{x-y}^2$, for all $x,y \in \X$.
\end{definition}

\begin{definition}[$\beta$-Smoothness] 
\label{def:smooth}
A function $f: \X \to \Real$ is $\beta$-\emph{smooth} w.r.t. a norm $\norm{\cdot}$ if it is differentiable and $\norm{\nabla f(x) - \nabla f(y)}_* \leq \beta \norm{x - y}$, for all $x,y \in \X$. 
\end{definition}

Next, we collect several assumptions which we will employ in the results to follow.
\begin{assumption}[Regularity assumptions]
\label{ass:regularity}
Let $\A$ be a Banach space equipped with the norm $\norm{\cdot}$. Suppose that
\begin{itemize}
    \item The set $\X$ is a convex subset of $\A$;
    \item The map $h : \A \to \Real$ is differentiable and $1$-strongly convex on $\X$;
    \item Each member of the cost sequence $\{ s_t \}_{t=1}^T$ is convex and $\beta$-smooth. Each member of the cost sequence $\{ r_t \}_{t=1}^T$ is convex;
    \item $\B_h(x,y) \leq R^2$ for all $x,y \in \X$, where $R>0$; 
    \item For all $t\in[T]$, the gradient prediction $\nabla \hat{s}_t$ satisfies $\norm{\nabla s_t(x) - \nabla \hat{s}_t(x)}_* \leq \sigma < \infty$ for any $x \in \X$;
    \item For all $t\in[T]$, the function prediction $\hat{r}_t$ is convex and $|r_t(x) - \hat{r}_t(x)| < \infty$ for any $x \in \X$.
\end{itemize}
\end{assumption}
In particular, the last two points of Assumption \ref{ass:regularity} simply state that the gradient and function predictions cannot be arbitrarily bad, which would naturally prevent the use of such predictive information. Next, we provide static and dynamic regret bounds that exploit gradient/function predictability and/or dynamical models of the reference sequence.

\subsection{Static Environments}
\label{subsec:main_results_static}

Our first result concerns convex costs in static environments. In order to exploit predictive information of composite costs of the form $f_t(\cdot) = s_t(\cdot) + r_t(\cdot)$, we propose the Optimistic Composite Mirror Descent (OptCMD) algorithm
\begin{align} \label{alg:OptCMD}
\tag{OptCMD}
\begin{aligned}
x_t &= \text{arg}\min_{x\in\X} \ \big\{ \eta_t \inner{\nabla \hat{s}_t(y_{t-1})}{x} + \eta_t\hat{r}_t(x)  + \B_h(x,y_{t-1})\big\} \\
y_t &= \text{arg}\min_{y\in\X} \ \big\{ \eta_t \inner{\nabla s_t(x_t)}{y} + \eta_t r_t(y)  + \B_h(y,y_{t-1})\big\},
\end{aligned}
\end{align}
where $\nabla \hat{s}_t$ is a gradient prediction of $\nabla \hat{s}_t$ and $\hat{r}_t$ is the function prediction of $r_t$. Notice that unlike algorithms \ref{alg:COMID} and \ref{alg:IOMD}, Algorithm \ref{alg:OptCMD} makes use of an auxiliary variable $y_t$. However, $x_t$ is still the decision variable of all OCO algorithms discussed in this paper. Algorithm \ref{alg:OptCMD} can be interpreted as an extension of \ref{alg:OptMD} for composite costs with smooth and nonsmooth components. As hinted in \ref{subsec:DT_problem}, one needs smooth functions to properly exploit gradient predictions of costs. Therefore, the intuition behind Algorithm \ref{alg:OptCMD} is similar to the one from proximal gradient algorithms: we handle nonsmooth components by not linearizing them in the proximal updates while linearizing the smooth ones. This leads to using \textit{function predictions} of the nonsmooth component $r_t$, instead of gradient predictions.

\begin{theorem}[Static regret: convex costs]\label{theorem:OptCMD_convex}
Suppose that Assumption \ref{ass:regularity} holds. Using the adaptive step-size 
$$
\eta_1 = \frac{1}{2\beta}, \quad \eta_t = \left( 4\beta^2 + \left(V_{t-1}' \right)^2 + D_{t-1}' \right)^{-\frac{1}{2}}
$$
for all $t>1$, Algorithm \ref{alg:OptCMD} guarantees
\begin{equation}
    \label{eq:stat_reg_bound}
    \Regret_T \leq O\left(1 + \sqrt{D_T'} + \min\left\{ V_T', \sqrt{T} \right\}\right).
\end{equation}
\end{theorem}

\begin{remark}[Intuition on adaptive step-size $\eta_t$]
In Theorem \ref{theorem:OptCMD_convex}, for simplicity, consider the case when $r_t(x) = 0 \ \forall x \in \X$, i.e., $V_T' = 0$. For this scenario, we want to guarantee $O(\sqrt{T})$ regret in the worst-case, and in order to do so, it is known we need $\eta_t = O(1/\sqrt{t})$. On the other hand, with perfect gradient predictions (i.e., $D'_T = 0$), we want to guarantee $O(1)$ regret, and in order to do so, we need $\eta_t \leq O(1/\beta)$ \cite{ho2019exploiting}. Now, if we want to guarantee a regret bound that generalizes these two extreme cases, it is natural that our step size should also generalize $\eta_t = O(1/\sqrt{t})$ and $\eta_t \leq O(1/\beta)$, which is precisely the behavior of the $\eta_t$ we designed. Similar intuitions can be derived from the other scenarios and results presented in this paper.
\end{remark}

The result of Theorem \ref{theorem:OptCMD_convex} is also related to \cite[Theorem 3]{mohri2016accelerating}, where the authors prove regret bounds for the so-called \textit{Composite Adaptive Optimistic Follow-the-Regularized-Leader} (CAO-FTRL) algorithm. The key differences between these results are: the CAO-FTRL algorithm uses FTRL update steps, which can be computationally more expensive than the mirror descent steps of Algorithm \ref{alg:OptCMD}; the CAO-FTRL algorithm assumes knowledge of $r_t$ at the beginning of round $t$, thus, is less general than Algorithm \ref{alg:OptCMD}; and finally, the regret bound of \cite[Theorem 3]{mohri2016accelerating} is presented in terms of $D_T$. Here we re-emphasize that we present regret bounds in terms of $D'_T$ instead of $D_T$, which solves the issues raised in Subsection \ref{subsec:DT_problem}. Key points to achieve this result are our proposed adaptive step-size (see remark above), and the extra assumption that the costs are $\beta$-smooth. In particular, since smooth costs have Lipschitz continuous gradients, we are able to control the difference between, possibly approximate, gradient predictions.

Next, we discuss how the bound of Theorem \ref{theorem:OptCMD_convex} generalizes several regret bounds from the literature.

\begin{remark}[Generality of regret bound]
\label{remark:OptCMD_convex}
First, let us consider the case when $r_t(x) = 0 \ \forall x \in \X$, i.e., $V_T' = 0$. In this case, $f_t$ is $\beta$-smooth convex, and Algorithm \ref{alg:OptCMD} reduces to Algorithm \ref{alg:OptMD}. In this scenario, when perfect predictions are available, setting $\nabla \hat{s}_t = \nabla s_t$ implies that $D'_T=0$, and the regret inequality \eqref{eq:stat_reg_bound} reduces to $\Regret_T \leq O(1)$, recovering the result of Ho-Nguyen and Kılınç-Karzan \cite{ho2019exploiting}.
On the other hand, in view of Assumption \ref{ass:regularity}, the regret inequality \eqref{eq:stat_reg_bound} also recovers the minimax static regret $\Regret_T \leq O(\sqrt{T})$ in the worst case, that is, even if the gradient predictions are completely uncorrelated with the true gradients and we end up with $D'_T = O(T)$. Next, consider the case when $s_t(x) = 0 \ \forall x \in \X$, i.e., $D_T' = 0$. In this case, $f_t$ is a general convex function and \eqref{eq:stat_reg_bound} reduces to $O\left(1 + \min\left\{ V_T', \sqrt{T} \right\}\right)$. Again, when perfect predictions are available we recover the optimal constant regret bound, by simply setting $\hat{r}_t = r_t$. This bound generalizes the $O\left(1 + \min\left\{ V_T, \sqrt{T} \right\}\right)$ bound of Campolongo and Orabona \cite{campolongo2020temporal}, which is known to be optimal \cite[Theorem 6.3]{campolongo2020temporal}. In this case, if our function predictions are good, $V_T'$ may be small and we guarantee small regret. On the other hand, we still guarantee the standard $O(\sqrt{T})$ regret in the worst-case. 
\end{remark}

Next, we state a static regret result for strongly convex costs.
This stronger assumption on the costs allows us to achieve tighter bounds. For this result, we need the following assumption. 

\begin{assumption}[Extra regularity assumptions]
\label{ass:extra_regularity}
Suppose that the action space $\A$ is an Euclidean space equipped with the $2$-norm $\norm{\cdot}_2$ and $h(x)=\frac{1}{2}\norm{x}^2_2$. Moreover, suppose we have access to perfect function prediction of $r_t$. That is, we are able to set $\hat{r}_t = r_t$ for all $t$.
\end{assumption}
    
Notice that under Assumption \ref{ass:extra_regularity}, the Bregman divergence $\B_h(x,y)=\frac{1}{2}\norm{x-y}^2_2$. Concerning the perfect prediction of $r_t$, this is the case, for example, when this term corresponds to a fixed known regularizer, e.g., $r_t(x) = \|x\|$, or naturally when $r_t(x) = 0$ for all $t$.

\begin{theorem}[Static regret: strongly convex costs]\label{theorem:OptCMD_strong}
Suppose that assumptions \ref{ass:regularity} and \ref{ass:extra_regularity} hold and that the costs $\{ s_t \}_{t=1}^T$ are $\alpha$-strongly convex. Using the adaptive step-size 
$$
\eta_1 = \frac{1}{2\beta}, \quad \eta_t = \left(2\beta +  \frac{\alpha}{2\sigma^2}D'_{t-1} \right)^{-1}
$$
for all $t>1$, Algorithm \ref{alg:OptCMD} with $\hat{r}_t = r_t$ guarantees
\begin{align}
\label{eq:stat_reg_inequality_SC}
\Regret_T \leq O\left(1 + \log(1+D'_T)\right). 
\end{align}
\end{theorem}

\begin{remark}[Generality of bound for strongly convex costs]\label{remark:OptCMD_strong}
Employing a similar line of argument as in Remark \ref{remark:OptCMD_convex}, we state two observations. 
With perfect gradient predictions, inequality \eqref{eq:stat_reg_inequality_SC} becomes $\Regret_T \leq O(1)$, again recovering the result of Ho-Nguyen and Kılınç-Karzan \cite{ho2019exploiting}. Moreover, the optimal regret bound $\Regret_T \leq O(\log(T))$ is also recovered in the worst-case.
\end{remark}

\subsection{Dynamic Environments}

As previously mentioned, when working in dynamic environments, we would like to exploit gradient predictions, function predictions, and \textit{knowledge of reference sequence dynamics}. Thus, in this scenario, we propose the \emph{Optimistic Dynamic Composite Mirror Descent} (OptDCMD) algorithm
\begin{align}  \label{alg:OptDCMD}
\tag{OptDCMD}
\begin{aligned}
x_t &= \text{arg}\min_{x\in\X} \ \big\{ \eta_t \inner{\nabla \hat{s}_t(y_{t-1})}{x} + \eta_t \hat{r}_t(x)  + \B_h(x,y_{t-1}) \big\} \\
\tilde{y}_t &= \text{arg}\min_{y\in\X} \ \big\{ \eta_t \inner{\nabla s_t(x_t)}{y} + \eta_t r_t(y)  + \B_h(y,y_{t-1}) \big\} \\
y_t &= \Phi_t(\tilde{y}_t).
\end{aligned}
\end{align}
This algorithm can be viewed as a combination of Algorithm \ref{alg:OptCMD} and the DMD algorithm of Hall and Willett \cite{hall2013dynamical}. To the best of our knowledge, no result in the literature has presented a regret analysis of an algorithm that combines gradient predictions, function predictions, and knowledge about the dynamics of the reference sequence. In what follows, we assume that the Player has access to dynamical models $\Phi_t$ of $\{u_t\}_{t=1}^T$. Let us further make the following assumptions.
\begin{assumption}[Lipschitz-likeness of $\B_h$]  \label{ass:lipschtz_breg}
For all $x,y,z \in \X$, there exist a scalar $\gamma > 0$ such that the Bregman divergence satisfies the Lipschitz-like condition $\B_h(x,z) - \B_h(y,z) \leq \gamma \norm{x-y}$.
\end{assumption}

\begin{remark}[Mildness of Assumption  \ref{ass:lipschtz_breg}]
It follows that Assumption \ref{ass:lipschtz_breg} holds when the mapping $h$ is Lipschitz on $\X$ \cite{jadbabaie2015online}, which is a mild assumption once $\mathcal{X}$ is usually a compact set. Some examples are $h(x) = \frac{1}{2} \|x\|^2_2$ (i.e., the euclidean case), or the ``KL divergence case'' [18].
\end{remark}

\begin{assumption}[Non-expansiveness of $\Phi_t$] \label{ass:nonexpansiveness_Phi}
For all $x,y \in \X$ and $\B_h$, the mapping $\Phi_t$ is non-expansive, that is, $\B_h(\Phi_t(x),\Phi_t(y)) - \B_h(x,y) \leq 0$.
\end{assumption}
\begin{remark}[Necessity of Assumption \ref{ass:nonexpansiveness_Phi}]
Observe that Assumption \ref{ass:nonexpansiveness_Phi} is a restriction on the class of dynamical models $\Phi_t$.
The reason behind this assumption is to control the impact of a possibly unreliable prediction (made by the use of $\Phi_t$), as the online game progresses \cite{hall2015online, shahrampour2017distributed}.
\end{remark}

We now present a dynamic regret bound for the Algorithm \ref{alg:OptDCMD}. In the results that follow, by abuse of notation, we use $V_t' := \sum_{\tau=1}^t | r_\tau(x_\tau) - \hat{r}_\tau(x_\tau) + \hat{r}_\tau(\tilde{y}_\tau) - r_\tau(\tilde{y}_\tau)|$.

\begin{theorem}[Dynamic regret: convex costs]\label{theorem:OptDCMD_convex}
Suppose that assumptions \ref{ass:regularity}, \ref{ass:lipschtz_breg} and \ref{ass:nonexpansiveness_Phi} hold. Define the adaptive step-size
$$
\eta_1 = \frac{1}{2\beta}, \quad \eta_t = \left( 4\beta^2 + \left(V_{t-1}' \right)^2 + D_{t-1}' \right)^{-\frac{1}{2}}
$$
for all $t>1$. Then, Algorithm \eqref{alg:OptDCMD} guarantees
\begin{equation}
    \label{eq:bound_thr_dyna_1}
    \begin{aligned}
        \DRegret_T \leq O\bigg((1 + C'_T )\Big(1 + \sqrt{D_T'} + \min\left\{ V_T', \sqrt{(1 + C'_T) T} \right\}\Big)\bigg).
    \end{aligned}
\end{equation}
\end{theorem}

\begin{remark}[Comparison with literature]
\label{remark:theorem_OptDCMD_convex}
Lets consider the case when $r_t(x) = 0 \ \forall x \in \X$, i.e., $V_T' = 0$. Observe that when $\Phi_t$ approximates the true dynamics of the comparator sequence $\{u_t\}_{t=1}^T$, we may have $C_T' \leq C_T$. Moreover, we also recover $C_T' = C_T$ if we choose $\Phi_t$ as the identity map. Therefore, compared to the bound $\DRegret_T \leq O((C_T+1)\sqrt{D_T + 1})$ of Jadbabaie et al. \cite{jadbabaie2015online}, our result improves it in the sense that it is given in terms of $C'_T$ and $D'_T$, instead of $C_T$ and $D_T$ (recall the discussion of Subsection \ref{subsec:DT_problem}). Moreover, recall that we have $\norm{\nabla s_t(y_{t-1}) - \nabla \hat{s}_t(y_{t-1})}_* \leq \sigma$ by Assumption \ref{ass:regularity}. Hence, it follows that ${O(\sqrt{1 + D'_T}(1 + C'_T)) = O(\sqrt{T}(1 + C'_T))}$ in the worst-case, and we recover the bound of Hall and Willett \cite{hall2013dynamical}. However, Zhang et al. \cite{zhang2018adaptive} proposed an algorithm called \emph{Ader}, which achieves the \emph{optimal} bound ${\DRegret_T \leq O\left(\sqrt{T(1 + C'_T)}\right)}$. Thus, in the worst-case, our regret bound does not recover the optimal one. Comparing \eqref{eq:bound_thr_dyna_1} with the $O(\min\{V_T,\sqrt{T(1 + \tau)}\})$ dynamic regret bound of Campolongo and Orabona \cite{campolongo2021closer}, where $\tau$ is a known upper bound of $C_T$, we see that \eqref{eq:bound_thr_dyna_1} has worst dependence of $C_T'$. This is mainly due to the fact that, in order to exploit gradient prediction, we need to have $\eta_t \leq 1/(2\beta)$. Since in \cite{campolongo2021closer} a fully implicit algorithm is used (see Algorithm \ref{alg:IOMD}), the step size can depend linearly on $\tau$, in other words, it can be as large as necessary.
\end{remark}

In the next theorem, we show that if the component $s_t = 0$, that is, $f_t = r_t$, we achieve a bound that generalizes the one from \cite{campolongo2021closer} using function predictions, i.e., using $V_T'$ instead of $V_T$. Notice that in this case, the updates of Algorithm \ref{alg:OptDCMD} are fully implicit updates, just like in Algorithm \ref{alg:IOMD}.

\begin{theorem}[Dynamic regret: implicit updates]\label{theorem:OptDCMD_implicit}
Suppose that assumptions \ref{ass:regularity}, \ref{ass:lipschtz_breg} and \ref{ass:nonexpansiveness_Phi} hold. Furthermore, let $s_t = 0$ for all $t$, and $\tau$ be an upper bound of $C'_T$. Define the adaptive step-size
$$
\eta_t = \frac{\tau}{V_{t-1}'}.
$$
for all $t>1$. Then, Algorithm \eqref{alg:OptDCMD} guarantees
\begin{equation}
    \label{eq:bound_thr_implicit}
    \begin{aligned}
        \DRegret_T \leq O \left( \min \left\{ V_T', \ \sqrt{(1 + \tau)T} \right\} \right).
    \end{aligned}
\end{equation}
\end{theorem}

As mentioned in Remark \ref{remark:theorem_OptDCMD_convex}, the Ader algorithm of Zhang et al. \cite{zhang2018adaptive} guarantees the optimal worst-case dynamic regret bound of ${\DRegret_T \leq O\left(\sqrt{T(1 + C'_T)}\right)}$, without prior knowledge of $C'_T$ or an upper bound on it. In order to achieve this bound, an expert-tracking algorithm based on Online Gradient Descent (OGD) updates is used. In our final result, we show that by using a step-size that adapts to $C'_t$, a similar regret bound can be achieved while also exploiting gradient predictions.

\begin{theorem}[Dynamic regret: fully adaptive step-size]
\label{theorem:OptDMD_convex_2}
Suppose that assumptions \ref{ass:regularity}, \ref{ass:lipschtz_breg} and \ref{ass:nonexpansiveness_Phi} hold. Furthermore, assume have access to $r_t$, thus, we can choose $\hat{r}_t = r_t$. Set the adaptive step-size to $\eta_1 = \eta_2 = 1/(2\beta)$ and 
$$
\eta_t = \sqrt{\frac{C'_{t-2} + 1}{D'_{t-1} + \theta_t}}
$$ 
for $t>2$, where $\theta_t$ is chosen such that $\eta_t \leq \eta_{t-1} \leq \frac{1}{2\beta}$ and $\theta_t \geq \theta_{t-1}$ for all $t$.
In this scenario, Algorithm \ref{alg:OptDCMD} guarantees
\begin{equation}
\label{eq:dyn_reg_inequality_2}
\DRegret_T \leq O\left(\sqrt{(\theta_T + D'_T)(1 + C'_T)}\right).
\end{equation}
\end{theorem}

\begin{remark}[Comments on $\theta_t$]
\label{remark:theta}
From the definition of the step-size used in Theorem \ref{theorem:OptDMD_convex_2}, we notice that the more the reference sequence $\{u_t\}_{t=1}^T$ varies (i.e. the bigger $C'_t$ is), the larger $\eta_t$ should be. Intuitively, we need larger step-sizes to ``track'' a reference sequence that changes a lot. On the other hand, in order to exploit gradient prediction, we also need $\eta_t \leq 1/(2\beta)$. Thus, $\theta_t$ can be interpreted as a trade-off parameter, which must be big enough so that $\eta_t \leq 1/(2\beta)$, but also not too big so that the algorithm is not able to ``track'' $\{u_t\}_{t=1}^T$. Also notice that, in the case of perfect gradient predictions (i.e. $D'_T = 0$), the regret bound of Theorem \ref{theorem:OptDMD_convex_2} becomes $\DRegret_T \leq O\left(1 + C'_T\right)$, since in this scenario we need $\theta_t = O\left(1 + C'_t\right)$ in order to guarantee that $\eta_t \leq 1/(2\beta)$. This dynamic regret bound is similar to the one presented in \cite{mokhtari2016online}, where the authors do not use any kind of gradient predictions, but assume \emph{strongly convex} costs and a specific reference sequence defined as $u_t = \emph{arg}\min_{x \in \X} f_t(x)$.
\end{remark}

In Theorem \ref{theorem:OptDMD_convex_2}, notice that feedback about $\|u_t - \Phi_{t-1}(u_{t-1})\|$ after round $t$ is necessary to implement the proposed step-size $\eta_t$. Although this information may not be available in the most general case of arbitrary costs $f_t$ and reference sequence $u_t$, it is reasonable to assume this type of feedback in many applications. For example, in the case where the reference sequence is a fixed point (and the dynamic regret reduces to static regret), the feedback assumption trivially holds since in this case $\|u_t - u_{t-1}\| = 0$. Another example is the case of quadratic costs (see experimental results in \cite{shahrampour2017distributed,mokhtari2016online}), which is ubiquitous in control applications. In this case, the gradient feedback $\nabla f_t$ constrains the information about $u_t$, thus, our step-sizes can be implemented. Finally, another common case is when $u_t = \text{arg}\min_{x \in \X} f_t(x)$. When the cost $f_t$ is revealed after round $t$, its optimizer can be computed, and again our step-sizes $\eta_t$ can be implemented, although it may be computationally expensive to do so.

Differently from the approach proposed in Theorem \ref{theorem:OptDMD_convex_2}, algorithms based on the doubling-trick or experts have been proposed as a way to adapt to $C'_T$ without knowing it in advance \cite{jadbabaie2015online, zhang2018adaptive, campolongo2021closer}. We leave it as an open question whether or not these tools can be used to prove tighter dynamic regret bounds when using gradient and/or function predictions. Moreover, our regret bounds can serve as the basis for the design and analyses of algorithms that learn gradient/function predictors and minimize regret simultaneously. For instance, in order to learn good predictors, it may be necessary to explore the action space by playing actions perturbed by some noise. This strategy may lead to regret bounds that depend on the prediction error (i.e, $D'_T$ and/or $V'_T$) and terms that depend on the perturbation noise. Studying the trade-off between exploration (playing perturbed action to learn good predictors and minimize $D'_T$ and/or $V'_T$) and exploitation (playing actions with low noise) is an interesting future work direction.

\section{Technical Proofs}
\label{sec:proofs}

We start this section with some auxiliary lemmas which will be useful in the proofs of our main results.

\subsection{Auxiliary Lemmas}

The following lemma is a straightforward generalization of the standard mirror descent inequality and is stated without proof.

\begin{lemma}
\label{lemma:main_inequality}
Suppose that $\X$ is a closed convex set. Let $\varphi : \X \to \Real$ be a convex function and $\eta >0$. Define
\begin{align*}
u := \arg\min_{x\in\X} \left\{ \eta \varphi(x) + \B_h(x,v) \right\}.
\end{align*}
It follows that, for all $ z \in \X$ and $g(u) \in \partial \varphi(u)$,
\begin{align*}
\eta\inner{g(u)}{u-z} \leq \B_h(z,v) - \B_h(z,u) - \B_h(u,v).
\end{align*}
\end{lemma}

The next lemma relates the proximal gradient updates (e.g. as in Algorithm \ref{alg:OptCMD}) with the gradients of the linearized components.
\begin{lemma} \label{lemma:OptCMD_diff}
Suppose that $\X$ is a closed convex set in a Banach space $\mathbb{S}$ equipped with a norm $\norm{\cdot}$. Let $h$ be 1-strongly convex w.r.t. $\|\cdot\|$. Let $w_1,w_2\in\mathbb{S}^*$, $v\in\X$, $r: \X \to \Real$ is a convex function and $\eta >0$. Define
$$
u_1 := \arg \min_{x_1\in\X} \Big\{ \inner{w_1}{x_1} + r(x_1) + \frac{1}{\eta}\B_h(x_1,v) \Big\},
$$
$$
u_2 := \arg \min_{x_2\in\X} \Big\{ \inner{w_2}{x_2} + r(x_2)  + \frac{1}{\eta}\B_h(x_2,v) \Big\}.
$$
Then, it holds that
\begin{align*}
\norm{u_1 - u_2} \leq \eta \norm{w_1 - w_2}_*.
\end{align*}
\end{lemma}
\begin{proof}
From the optimality of $u_1$ and $u_2$ \cite[Theorem 3.1.24]{nesterov2018lectures}, we have
$$
\eta\inner{w_1 + g(u_1)}{u_2-u_1} \geq \inner{\nabla h(u_1) - \nabla h(v)}{u_1-u_2},
$$
and
$$
\eta\inner{-w_2 - g(u_2)}{u_2-u_1} \geq \inner{\nabla h(v) - \nabla h(u_2)}{u_1-u_2}.
$$
where $g(u) \in \partial r(u) \ \forall u \in \X$. Adding these two inequalities up, we get
\begin{equation}
\label{eq:breg_cos1}
\eta \inner{w_1 - w_2 + g(u_1) - g(u_2)}{u_2-u_1} \geq \inner{\nabla h(u_1) - \nabla h(u_2)}{u_1-u_2}.
\end{equation}
Since $h$ is $1$-strongly convex, it follows that
\begin{equation}
\label{eq:breg_cos2}
\inner{\nabla h(u_1) - \nabla h(u_2)}{u_1-u_2} \geq \norm{u_1-u_2}^2.
\end{equation}
Combining \eqref{eq:breg_cos1} and \eqref{eq:breg_cos2}, using the Cauchy-Schwarz inequality and the monotonicity of the subgradient $\inner{g(u_1) - g(u_2)}{u_2-u_1} \leq 0$, we have 
$$
\norm{u_1-u_2}^2 \leq \eta \norm{w_1 - w_2}_*\norm{u_2-u_1}.
$$
As a result, the claim follows. 
\end{proof}

The next lemma is useful for upper bounding quantities arising from the use of adaptive step-sizes in OCO algorithms.
\begin{lemma} \label{lemma:sqrt_bound_2}
Let $\{a_k\}_{k=1}^{T}$, $\{b_k\}_{k=1}^{T+1}$, $\{c_k\}_{k=1}^T$, be nonnegative sequences, with $b_{t+1} \geq b_t$ and $c_{t+1} \geq c_t$. Then, for $T \geq 1$,
$$
\sum_{t=1}^T a_t \sqrt{\frac{b_t}{c_t + \sum_{k=1}^t a_k}} \leq 2\sqrt{b_T\left(c_T + \sum_{t=1}^T a_t\right)}.
$$
\end{lemma}
\begin{proof}
The proof is by induction. For $T=1$, one can show analytically that the inequality holds. Suppose that the inequality holds for some $T-1>2$. Thus, it follows that
\begin{align*}
\sum_{t=1}^T a_t \sqrt{\frac{b_t}{c_t + \sum_{k=1}^t a_k}} 
&= \sum_{t=1}^{T-1} a_t \sqrt{\frac{b_t}{c_t + \sum_{k=1}^t a_k}} +  a_T \sqrt{\frac{b_T}{c_T + \sum_{k=1}^T a_k}} \\ 
&\leq 2\sqrt{b_{T-1}\left(c_{T-1} + \sum_{t=1}^{T-1} a_t\right)} + a_T \sqrt{\frac{b_T}{c_T + \sum_{k=1}^T a_k}} \\
&\leq 2\sqrt{b_T\left(c_T - a_T + \sum_{t=1}^{T} a_t\right)} + a_T \sqrt{\frac{b_T}{c_T + \sum_{k=1}^T a_k}} \\
&= \sqrt{b_T}\left(2\sqrt{A - a_T} + \frac{a_T}{\sqrt{A}}\right),
\end{align*}
where $A := c_T + \sum_{t=1}^T a_t$. As a function of $a_T \geq 0$, one can show that the R.H.S of the previous inequality is maximized when $a_T = 0$. Thus,
$$
\sqrt{b_T}\left(2\sqrt{A - a_T} + \frac{a_T}{\sqrt{A}}\right) \leq 2\sqrt{b_TA}.
$$
This concludes the proof.
\end{proof}

The next lemma is useful for upper bounding quantities arising from the use of adaptive step-sizes in OCO algorithms, especially when the costs are strongly convex.
\begin{lemma} \label{lemma:log_bound}
Given two positive reals $a$ and $b$, it holds that
\begin{align*}
b\left( \frac{1}{b} - \frac{1}{a} \right) \leq \log \left( \frac{b^{-1}}{a^{-1}} \right). 
\end{align*}
\end{lemma}
\begin{proof}
Let us first recall the identity $\log(\xi) \leq \xi-1$, for any $\xi>0$. 
Set $\xi = a^{-1}/b^{-1}$. Notice that
\begin{align*}
    -\log\left(\frac{b^{-1}}{a^{-1}}\right) = \log\left(\frac{a^{-1}}{b^{-1}}\right) \leq \frac{a^{-1}}{b^{-1}}-1=b\left( \frac{1}{a} - \frac{1}{b} \right). 
\end{align*}
Thus, the claim is an immediate consequence of the above relation. 
\end{proof}

\subsection{Main Proofs}

Next, we continue with the proofs of our main results.

\subsubsection{Proof of Theorem \ref{theorem:OptCMD_convex}}

Define $x^* := \text{arg}\min_{x \in \X} \sum_{t=1}^T f_t(x)$. From the definition of $f_t$,
\begin{align*}
    f_t(x_t) - f_t(x^*)
    &= s_t(x_t) + r_t(x_t) - s_t(x^*) - r_t(x^*) \\
    &= s_t(x_t) - s_t(x^*) + r_t(x_t) - \hat{r}_t(x_t) + \hat{r}_t(x_t) - r_t(y_t) + r_t(y_t) - r_t(x^*) \\
    &\leq  \Delta_t + \inner{\nabla s_t(x_t)}{x_t - x^*} + \inner{\hat{g}_t(x_t)}{x_t - y_t} + \inner{g_t(y_t)}{y_t - x^*} \\
    &= \Delta_t + \inner{\nabla s_t(x_t) - \nabla \hat{s}_t(y_{t-1})}{x_t - y_t} + \inner{\nabla \hat{s}_t(y_{t-1}) + \hat{g}_t(x_t)}{x_t - y_t} \\
    &\quad + \inner{\nabla s_t(x_t) + g_t(y_t)}{y_t - x^*},
\end{align*}
where $g_t(y_t) \in \partial r_t(y_t)$, $\hat{g}_t(x_t) \in \partial \hat{r}_t(x_t)$, $\Delta_t := r_t(x_t) - \hat{r}_t(x_t) + \hat{r}_t(y_t) - r_t(y_t)$ and the inequality follows from the convexity of $s_t$, $r_t$ and $\hat{r}_t$. Using Lemma \ref{lemma:main_inequality}, we get 
\begin{align*}
    f_t(x_t) - f_t(x^*) 
    &\leq \Delta_t + \inner{\nabla s_t(x_t) - \nabla \hat{s}_t(y_{t-1})}{x_t - y_t} + \frac{1}{\eta_t}\big(\B_h(x^*,y_{t-1}) - \B_h(x^*,y_t) \\
    &\quad - \B_h(x_t,y_{t-1}) - \B_h(y_t,x_t)\big) \\
    &\leq \Delta_t + \|\nabla s_t(x_t) - \nabla \hat{s}_t(y_{t-1})\|_*\|x_t - y_t\| + \frac{1}{\eta_t}\big(\B_h(x^*,y_{t-1}) - \B_h(x^*,y_t) \\
    &\quad - \B_h(x_t,y_{t-1}) - \B_h(y_t,x_t)\big) \\
    &= A_t + B_t + C_t,
\end{align*}
where $A_t$, $B_t$ and $C_t$ are defined as
\begin{align*}
    A_t &:= \|\nabla s_t(x_t) - \nabla \hat{s}_t(y_{t-1})\|_*\|x_t - y_t\| - \frac{1}{2\eta_t}\B_h(y_t,x_t) - \frac{1}{\eta_t}\B_h(x_t,y_{t-1}),
\end{align*}
$$
B_t := \frac{1}{\eta_t}\big( \B_h(x^*,y_{t-1}) - \B_h(x^*,y_t) \big) \quad \text{and} \quad  C_t := \Delta_t- \frac{1}{2\eta_t}\B_h(y_t,x_t).
$$
We will proceed by upper bounding $\sum_{t=1}^T A_t$ and $\sum_{t=1}^T B_t$ separately.

\textbf{(Upper bounding $\sum_{t=1}^T A_t$)}

Starting from the fact that $-\B_h(x,y) \leq - \frac{1}{2}\norm{x-y}^2$ and that $ab \leq \rho a^2 + \frac{b^2}{4\rho}$ for any $\rho > 0$, we have
\begin{align}
    \label{eq:bound_sum_A_mirror}
    A_t
    &= \|\nabla s_t(x_t) - \nabla \hat{s}_t(y_{t-1})\|_*\|x_t - y_t\| - \frac{1}{2\eta_t}\B_h(y_t,x_t) - \frac{1}{\eta_t}\B_h(x_t,y_{t-1}) \nonumber \\
    &\leq \|\nabla s_t(x_t) - \nabla \hat{s}_t(y_{t-1})\|_*\|x_t - y_t\| - \frac{1}{4\eta_t}\|y_t - x_t\|^2 - \frac{1}{2\eta_t}\norm{x_t-y_{t-1}}^2 \nonumber \\
    &\leq \eta_{t+1}\|\nabla s_t(x_t) - \nabla \hat{s}_t(y_{t-1})\|_*^2 + \left( \frac{1}{4\eta_{t+1}} - \frac{1}{4\eta_t} \right)\norm{x_t-y_t}^2 - \frac{1}{2\eta_t}\norm{x_t-y_{t-1}}^2 \nonumber \\
    &\leq 2\eta_{t+1}\|\nabla s_t(x_t) - \nabla s_t(y_{t-1})\|_*^2 - \frac{1}{2\eta_t}\norm{x_t-y_{t-1}}^2  + 2\eta_{t+1}\|\nabla s_t(y_{t-1}) - \nabla \hat{s}_t(y_{t-1})\|_*^2 +  \frac{R^2}{2\eta_{t+1}} - \frac{R^2}{2\eta_t} \nonumber \\
    &\leq \left(2\beta^2\eta_{t+1} - \frac{1}{2\eta_t} \right)\norm{x_t-y_{t-1}}^2 + \frac{R^2}{2\eta_{t+1}} - \frac{R^2}{2\eta_t} + 2\eta_{t+1}\|\nabla s_t(y_{t-1}) - \nabla \hat{s}_t(y_{t-1})\|_*^2 \nonumber \\
    &\leq 2\eta_{t+1}\|\nabla s_t(y_{t-1}) - \nabla \hat{s}_t(y_{t-1})\|_*^2 + \frac{R^2}{2\eta_{t+1}} - \frac{R^2}{2\eta_t},
\end{align}
where used the facts that $\eta_t$ is nonincreasing, Assumption \ref{ass:regularity} and the fact that $\eta_t \leq \frac{1}{2\beta}$, which implies $2\beta^2\eta_{t+1} - \frac{1}{2\eta_t} \leq 0$. 
Next, we will bound the two terms of \eqref{eq:bound_sum_A_mirror} separately. Summing the first term over $t=1,\ldots,T$, we get
\begin{align*}
    2\sum_{t=1}^T \eta_{t+1}\|\nabla s_t(y_{t-1}) - \nabla \hat{s}_t(y_{t-1})\|_*^2 
    &= 2\sum_{t=1}^T \frac{\|\nabla s_t(y_{t-1}) - \nabla \hat{s}_t(y_{t-1})\|_*^2}{\sqrt{4\beta^2 + \left(V_t' \right)^2 + D'_t}} \leq 4\sqrt{4\beta^2 + \left(V_T' \right)^2 + D'_T} \\
    &\leq 4V_T' + 4\sqrt{4\beta^2 + D_T'},
\end{align*}
where the inequalities follow from the definition of $D_t'$, Lemma \ref{lemma:sqrt_bound_2} and $\sqrt{a + b} \leq \sqrt{a} + \sqrt{b}$. Summing the second and third terms of \eqref{eq:bound_sum_A_mirror} over $t=1,\ldots,T$ and telescoping the sum, we get
\begin{align*}
    \frac{R^2}{2}\sum_{t=1}^T \left( \frac{1}{\eta_{t+1}} - \frac{1}{\eta_t} \right) \leq \frac{R^2}{2\eta_{T+1}}.
\end{align*}
Putting these bounds together, we arrive at
\begin{align}
    \label{eq:A_bound_mirror}
    \sum_{t=1}^T A_t 
    &\leq 4V_T' + \frac{R^2}{2\eta_{T+1}} + 4\sqrt{4\beta^2 + D_T'} = \left(4 + \frac{R^2}{2} \right)\left(V_T' + \sqrt{4\beta^2 + D_T'} \right).
\end{align}

\textbf{(Upper bounding $\sum_{t=1}^T B_t$)}

Rearranging and telescoping the sum, we have
\begin{align}
    \label{eq:B_bound_mirror}
    \sum_{t=1}^T B_t 
    \leq \sum_{t=1}^T \frac{1}{\eta_t}\big( \B_h(x^*,y_{t-1}) - \B_h(x^*,y_t) \big) \leq \frac{\B_h(x^*,y_0)}{\eta_1} + \sum_{t=1}^{T-1}\left(\frac{1}{
    \eta_{t+1}} - \frac{1}{
    \eta_t} \right) \B_h(x^*,y_t) \leq \frac{R^2}{\eta_T},
\end{align}
where we used Assumption \ref{ass:regularity}. Putting \eqref{eq:A_bound_mirror} and \eqref{eq:B_bound_mirror} together, we arrive at
\begin{align}
    \Regret_T 
    &\leq \sum_{t=1}^T \left( A_t + B_t + C_t \right) \nonumber \\
    &\leq  \left(4 + \frac{R^2}{2} \right)\left(V_T' + \sqrt{4\beta^2 + D_T'} \right) + \frac{R^2}{\eta_T} + \sum_{t=1}^T \left(\Delta_t- \frac{1}{2\eta_t}\B_h(y_t,x_t)\right) \nonumber \\
    &\leq  \left(4 + \frac{3R^2}{2} \right)\left(V_T' + \sqrt{4\beta^2 + D_T'} \right) + \sum_{t=1}^T \left(\Delta_t- \frac{1}{2\eta_t}\B_h(y_t,x_t)\right) \label{eq:bound_12_mirror} \\
    &\leq \left(5 + \frac{3R^2}{2} \right)\left(V_T' + \sqrt{4\beta^2 + D_T'} \right) \label{eq:bound_1_mirror},
\end{align}
where we used the definition of $\eta_{T+1}$ and the fact that $V_T' = \sum_{t=1}^T |\Delta_t|$ by definition. Similar to \cite[Theorem 6.2]{campolongo2020temporal}, we will proceed to bound the regret in a second way, which in turn will imply the regret is upper bounded by the minimum of \eqref{eq:bound_1_mirror} and the second bound. In particular, we will focus on the following part of \eqref{eq:bound_12_mirror}
\begin{align*}
    \left( 5 + \frac{3R^2}{4} \right)V_T' - \sum_{t=1}^T \frac{1}{2\eta_t}\B_h(y_t,x_t) = \left( 5 + \frac{3R^2}{4} \right)\sum_{t=1}^T |\Delta_t| - \sum_{t=1}^T \frac{1}{4\eta_t}\|x_t - y_t\|^2 = c\lambda_T,
\end{align*}
where $\lambda_T := \sum_{t=1}^T \left( |\Delta_t| -  \frac{1}{4c\eta_t}\|x_t - y_t\|^2\right)$ and $c := 5 + \frac{3R^2}{4}$. Thus,
\begin{equation}
    \label{eq:bound_2_1_mirror}
    \Regret_T \leq c\lambda_T + \left( 4 + \frac{3R^2}{4} \right)\sqrt{4\beta^2 + D_T'}.
\end{equation}
Next, we will proceed to prove an upper bound to $\lambda_T^2$, which will naturally imply an upper bound to $c\lambda_T$. To do so, we will first prove an upper bound to the term $|\Delta_t| -  \frac{1}{4c\eta_t}\|x_t - y_t\|^2$.

\textbf{(Upper bounding $|\Delta_t| -  \frac{1}{4c\eta_t}\|x_t - y_t\|^2$)}

From the definition of $\Delta_t$ and convexity of $r_t$ and $\hat{r}_t$, we have that
\begin{align}
    \label{eq:ineq1_mirror}
    \Delta_t 
    &\leq \inner{g_t(x_t) - \hat{g}_t(y_t)}{x_t - y_t} \leq \sqrt{2}R\|g_t(x_t) - \hat{g}_t(y_t)\|_*  + \frac{\norm{x_t-y_t}^2}{4c\eta_t}
\end{align}
and
\begin{align}
    \label{eq:ineq2_mirror}
    \Delta_t 
    &\leq \|g_t(x_t) - \hat{g}_t(y_t)\|_*\|x_t - y_t\| \leq c\eta_t\|g_t(x_t) - \hat{g}_t(y_t)\|^2_* + \frac{\norm{x_t-y_t}^2}{4c\eta_t}
\end{align}
for any $g_t(x_t) \in \partial r_t(x_t)$ and $\hat{g}_t(y_t) \in \partial \hat{r}_t(y_t)$, where we used the facts that $\frac{\norm{x_t-y_t}^2}{4c\eta_t} \geq 0$ and $ ab \leq \rho a^2 + \frac{b^2}{4\rho}$ for any $\rho > 0$. Similarly, we also have that 
\begin{equation}
    \label{eq:ineq3_mirror}
    -\Delta_t \leq \sqrt{2}R\|g_t(y_t) - \hat{g}_t(x_t)\|_*  + \frac{\norm{x_t-y_t}^2}{4c\eta_t}
\end{equation}
and
\begin{equation}
    \label{eq:ineq4_mirror}
    -\Delta_t \leq c\eta_t\|g_t(y_t) - \hat{g}_t(x_t)\|^2_* + \frac{\norm{x_t-y_t}^2}{4c\eta_t},
\end{equation}
Combining \eqref{eq:ineq1_mirror}, \eqref{eq:ineq2_mirror}, \eqref{eq:ineq3_mirror} and \eqref{eq:ineq4_mirror}, we get that
\begin{equation}
    \label{eq:bound_abs_delta_mirror}
    |\Delta_t| - \frac{\norm{x_t-y_t}^2}{4c\eta_t} \leq \min \left\{\sqrt{2}RG_t, c\eta_tG_t^2 \right\},
\end{equation}
where $G_t := \max \left\{\|g_t(x_t) - \hat{g}_t(y_t)\|_*, \|g_t(y_t) - \hat{g}_t(x_t)\|_* \right\}$.

\textbf{(Upper bounding $\lambda_T^2$)}

Notice that
\begin{equation}
    \label{eq:comp_min_upper_bound_mirror}
    \lambda_t - \lambda_{t-1} = |\Delta_t| - \frac{\norm{x_t-y_t}^2}{4c\eta_t} \leq \min \left\{\sqrt{2}RG_t, c\eta_tG_t^2 \right\}.
\end{equation}
Defining $\lambda^2_0 := 0$, we have that
\begin{align*}
    \lambda^2_T
    &= \sum_{t=1}^T \left( \lambda^2_t - \lambda^2_{t-1} \right) = \sum_{t=1}^T \left(\left( \lambda_t - \lambda_{t-1} \right)^2 + 2\left( \lambda_t - \lambda_{t-1} \right)\lambda_{t-1}\right) \leq \sum_{t=1}^T \left(2R^2G_t^2 + 2c\eta_t\lambda_{t-1}G_t^2\right),
\end{align*}
where the last inequality follows from \eqref{eq:comp_min_upper_bound_mirror}. Next, notice that by definition
$$
\eta_t\lambda_{t-1} = \frac{\sum_{k=1}^{t-1}\left( |\Delta_k| - \frac{\norm{x_k-y_k}^2}{2c\eta_k}\right)}{\sqrt{4\beta^2 + \left( \sum_{k=1}^{t-1} |\Delta_k| \right)^2 + D_{t-1}'}} \leq \frac{\sum_{k=1}^{t-1} |\Delta_k|}{\sum_{k=1}^{t-1} |\Delta_k|} = 1.
$$
Thus, we have that
$$
\lambda^2_T \leq \sum_{t=1}^T \left(2R^2G_t^2 + 2cG_t^2\right) = (2R^2+2c)\sum_{t=1}^T G_t^2.
$$
Taking the square root and substituting it into \eqref{eq:bound_2_1_mirror}, we get the second regret bound
\begin{equation}
    \label{eq:bound_2_mirror}
    \Regret_T \leq c\sqrt{2R^2+2c}\sqrt{\sum_{t=1}^T G_t^2} + \left( 5 + \frac{3R^2}{4} \right)\sqrt{4\beta^2 + D_T'}.
\end{equation}
Finally, combining \eqref{eq:bound_1_mirror} and \eqref{eq:bound_2_mirror}, arrive at
\begin{align*}
    \Regret_T 
    &\leq \min\left\{ \left( 5 + \frac{3R^2}{4} \right)V_T', \  c\sqrt{2R^2+2c}\sqrt{\sum_{t=1}^T G_t^2} \right\} + \left( 5 + \frac{3R^2}{4} \right)\sqrt{4\beta^2 + D_T'} \\
    &= O\left(1 + \sqrt{D_T'} + \min\left\{ V_T', \sqrt{T} \right\}\right).
\end{align*}
This completes the proof. $\Box$

\subsubsection{Proof of Theorem \ref{theorem:OptCMD_strong}}

In order to prove Theorem \ref{theorem:OptCMD_strong}, first we will prove a version of this theorem for general Bregman divergences and a general notion of strong convexity (Lemma \ref{lemma:OptCMD_strong_general}). This result is achieved by exploiting a certain technical assumption (Assumption \ref{ass:bregman_reg}). Then, we will show that for the euclidean case (i.e. $\B_h(x,y) = \frac{1}{2}\norm{x-y}^2_2$), this technical assumption always holds and Theorem \ref{theorem:OptCMD_strong} follows.

\begin{definition}[$\alpha$-Strong convexity w.r.t. $\B_h$]\label{def:breg_strong_convexity}
A function $f: \X \to \Real$ is $\alpha$-\emph{strongly convex} w.r.t. $\B_h$ if $f(x) - f(y) \leq \inner{\nabla f(x)}{x-y} - \alpha \B_h(y,x)$, for all $x,y \in \X$.
\end{definition}

\begin{assumption}[Technical assumption]
\label{ass:bregman_reg}
For $1/\eta > \alpha > 0$, there exists a constant $\lambda>0$ such that $\lambda \B_h(x,y) -  \frac{1}{\eta}\B_h(y,z) - \alpha \B_h(x,z) \leq 0$, for all $x,y,z \in \X$.
\end{assumption}

Before stating the general version of Theorem \ref{theorem:OptCMD_strong}, we make a short remark on Assumption \ref{ass:bregman_reg}.

\begin{remark}[Mildness of Assumption \ref{ass:bregman_reg}]
Notice that $\eta$, $\alpha$ and $\B_h(x,y)$ are all non-negative. Thus, for a general choice of $h$, one should expect to be able to choose a small enough $\lambda$ to ensure that the inequality in Assumption \ref{ass:bregman_reg} holds. In particular, when $\B_h(x,y) = \frac{1}{2}\norm{x-y}^2_2$, we will show that Assumption \ref{ass:bregman_reg} holds for $\lambda=\alpha/2$.
\end{remark}

\begin{lemma}[Strongly convex case with general divergence]\label{lemma:OptCMD_strong_general}
Suppose that Assumptions \ref{ass:regularity} and \ref{ass:bregman_reg} hold and that the costs $\{ f_t \}_{t=1}^T$ are $\alpha$-strongly convex w.r.t. $B_h$.
Using the adaptive step-size
$$
\eta_1 = \frac{1}{2\beta}, \quad \eta_t = \left( \frac{\lambda}{\sigma^2}D'_{t-1} + 2\beta \right)^{-1}
$$
for all $t>1$, Algorithm \ref{alg:OptCMD} with $\hat{r}_t = r_t$ guarantees
$$
\Regret_T \leq O\left(1 + \log(1+D'_T)\right).
$$
\end{lemma}
\begin{proof}
Let $x^* := \text{arg}\min_{x \in \X} \sum_{t=1}^T f_t(x)$. Since $s_t$ is $\alpha$-strongly convex w.r.t. $\B_h$, we have
$$
s_t(x_t) - s_t(x^*) \leq \inner{\nabla s_t(x_t)}{x_t - x^*} - \alpha \B_h(x^*,x_t).
$$
Thus,
\begin{align*}
    f_t(x_t) - f_t(x^*)
    &= s_t(x_t) + r_t(x_t) - s_t(x^*) - r_t(x^*) \\
    &= s_t(x_t) - s_t(x^*) + r_t(x_t) - r_t(x_t)  + r_t(x_t) - r_t(y_t) + r_t(y_t) - r_t(x^*) \\
    &\leq  \inner{\nabla s_t(x_t)}{x_t - x^*} - \alpha \B_h(x^*,x_t) + \inner{g_t(x_t)}{x_t - y_t} + \inner{g_t(y_t)}{y_t - x^*} \\
    &= \inner{\nabla s_t(x_t) - \nabla \hat{s}_t(y_{t-1})}{x_t - y_t} + \inner{\nabla \hat{s}_t(y_{t-1}) + g_t(x_t)}{x_t - y_t} \\
    &\quad + \inner{\nabla s_t(x_t) + g_t(y_t)}{y_t - x^*} - \alpha \B_h(x^*,x_t),
\end{align*}
where $g_t(y_t) \in \partial r_t(y_t)$, $g_t(x_t) \in \partial r_t(x_t)$ and the inequality follows from the convexity of $s_t$ and $r_t$. Using Lemma \ref{lemma:main_inequality}, we get 
\begin{align*}
    f_t(x_t) - f_t(x^*) 
    &\leq \inner{\nabla s_t(x_t) - \nabla \hat{s}_t(y_{t-1})}{x_t - y_t} - \alpha \B_h(x^*,x_t) + \frac{1}{\eta_t}\big(\B_h(x^*,y_{t-1}) \\
    &\quad - \B_h(x^*,y_t) - \B_h(x_t,y_{t-1}) - \B_h(y_t,x_t)\big) \\
    &\leq \|\nabla s_t(x_t) - \nabla \hat{s}_t(y_{t-1})\|_*\|x_t - y_t\| - \alpha \B_h(x^*,x_t) + \frac{1}{\eta_t}\big(\B_h(x^*,y_{t-1}) \\
    &\quad - \B_h(x^*,y_t) - \B_h(x_t,y_{t-1}) - \B_h(y_t,x_t)\big) \\
    &= A_t + B_t,
\end{align*}
where $A_t$ and $B_t$ are defined as
$$
A_t := \frac{1}{\eta_t} \big( \B_h(x^*,y_{t-1}) - \B_h(x^*,y_t) - \B_h(y_t,x_t) \big) - \alpha \B_h(x^*,x_t)
$$
and
$$
B_t := \|\nabla s_t(x_t) - \nabla \hat{s}_t(y_{t-1})\|_*\|x_t - y_t\| - \frac{1}{\eta_t}\B_h(x_t,y_{t-1}).
$$
With the above notations at hand, it follows that
\begin{align}
\label{eq:Thm3-reg_AtBt_composite}
\Regret_T = \sum_{t=1}^T \left( f_t(x_t) - f_t(x^*) \right) \leq \sum_{t=1}^TA_t + \sum_{t=1}^TB_t.
\end{align}
We proceed by bounding $\sum_{t=1}^TA_t$ and $\sum_{t=1}^TB_t$  separately. 

\textbf{(Upper bounding $\sum_{t=1}^T A_t$)} Observe that
\begin{align*}
\sum_{t=1}^T A_t 
&= \sum_{t=1}^T \frac{1}{\eta_t} \big( \B_h(x^*,y_{t-1}) - \B_h(x^*,y_t) \big) - \sum_{t=1}^T \left(\frac{1}{\eta_t} \B_h(y_t,x_t) + \alpha \B_h(x^*,x_t) \right) \\
&\leq \frac{\B_h(x^*,y_0)}{\eta_1} + \sum_{t=1}^T \left( \frac{1}{\eta_{t+1}} - \frac{1}{\eta_t} \right) \B_h(x^*,y_t) - \sum_{t=1}^T \left(\frac{1}{\eta_t} \B_h(y_t,x_t) + \alpha \B_h(x^*,x_t) \right).
\end{align*}
Assumption \ref{ass:regularity} and $\eta_1= \frac{1}{2\beta}$ imply that 
\begin{subequations}
\label{eq:Thm3-At_uppers}
\begin{align}
\label{eq:Thm3-At_1st-term}
    \frac{\B_h(x^*,y_0)}{\eta_1} \leq 2\beta R^2.
\end{align}
From the definition of $\eta_t$, we have that 
\begin{align}
\label{eq:Thm3-At_2nd-term}
    \frac{1}{\eta_{t+1}} - \frac{1}{\eta_t} = \frac{\lambda}{\sigma^2} \norm{\nabla s_t(y_{t-1}) - \nabla \hat{s}_t(y_{t-1})}^2_*.
\end{align}
Hence, we obtain
\begin{align}
\label{eq:Thm3-At_3rd-term}
    \sum_{t=1}^T \left( \frac{1}{\eta_{t+1}} - \frac{1}{\eta_t} \right)\B_h(x^*,y_t) 
    &= \lambda \sum_{t=1}^T  \frac{\norm{\nabla s_t(y_{t-1}) - \nabla \hat{s}_t(y_{t-1})}^2_*}{\sigma^2} \B_h(x^*,y_t) \leq \lambda \sum_{t=1}^T \B_h(x^*,y_t),
\end{align}
\end{subequations}
where the inequality follows from the fifth item in Assumption \ref{ass:regularity}. In light of the upper bounds derived in equation \eqref{eq:Thm3-At_uppers}, we then infer that
\begin{align}
\label{eq:Th3-At}
   \sum_{t=1}^T A_t 
    &\leq 2\beta R^2 + \sum_{t=1}^T\Big( \lambda\B_h(x^*,y_t) - \frac{1}{\eta_t} \B_h(y_t,x_t) - \alpha \B_h(x^*,x_t) \Big) \leq 2\beta R^2,
\end{align}
where the second inequality follows from Assumption \ref{ass:bregman_reg}. 

\textbf{(Upper bounding $\sum_{t=1}^T B_t$)}
Invoking Lemma \ref{lemma:OptCMD_diff}, we conclude that
$$
\norm{y_t - x_t} \leq \eta_t \norm{\nabla s_t(x_t) - \nabla \hat{s}_t(y_{t-1})}_*,
$$
and as a result,
$$
B_t \leq \eta_t \norm{\nabla s_t(x_t) - \nabla \hat{s}_t(y_{t-1})}_*^2  - \frac{1}{\eta_t}\B_h(x_t,y_{t-1}).
$$
Notice that
\begin{align*}
    B_t 
    &\leq 2\eta_t \norm{\nabla s_t(x_t) - \nabla s_t(y_{t-1})}^2_* + 2\eta_t\norm{\nabla s_t(y_{t-1}) - \nabla \hat{s}_t(y_{t-1})}^2_*  - \frac{1}{\eta_t} \B_h(x_t,y_{t-1}) \\
    &\leq 2\eta_t \beta^2 \norm{x_t - y_{t-1}}^2 + 2\eta_t\norm{\nabla s_t(y_{t-1}) - \nabla \hat{s}_t(y_{t-1})}^2_* - \frac{1}{\eta_t} \B_h(x_t,y_{t-1}).
\end{align*}
where we made use of the identity $\norm{a-b}^2 \leq 2\norm{a-c}^2 + 2\norm{c-b}^2$ and the $\beta$-smoothness of $s_t$. Using Lemma $-\B_h(x,y) \leq - \frac{1}{2}\norm{x-y}^2$, we arrive at
\begin{align*}
    B_t & \leq \left( 2\eta_t \beta^2 - \frac{1}{2\eta_t} \right) \norm{x_t - y_{t-1}}^2 + 2\eta_t\norm{\nabla s_t(y_{t-1}) - \nabla \hat{s}_t(y_{t-1})}^2_*.
\end{align*}
From the definition of $\eta_t$, we have that $2\eta_t \beta^2 - \frac{1}{2\eta_t} \leq 0 \ \forall t\geq 1$, and summing $B_t$ over $t=1,\ldots,T$ yields
\begin{align*}
    \sum_{t=1}^T B_t 
    &\leq  2\sum_{t=1}^T \eta_t\norm{\nabla s_t(y_{t-1}) - \nabla \hat{s}_t(y_{t-1})}^2_* \\
    &\leq 2\sum_{t=1}^T \eta_{t+1}\norm{\nabla s_t(y_{t-1}) - \nabla \hat{s}_t(y_{t-1})}^2_* + 2\sum_{t=1}^T (\eta_t-\eta_{t+1})\norm{\nabla s_t(y_{t-1}) - \nabla \hat{s}_t(y_{t-1})}^2_*.
\end{align*}
By virtue of the fifth item in Assumption \ref{ass:regularity}, it follows that
\begin{align*}
    \sum_{t=1}^T (\eta_t-\eta_{t+1})\norm{\nabla s_t(y_{t-1}) - \nabla \hat{s}_t(y_{t-1})}^2_* \leq \sigma^2\sum_{t=1}^T (\eta_t-\eta_{t+1}) = \sigma^2 \left( \eta_1 - \eta_{T+1} \right) \leq \sigma^2 \eta_1 = \frac{\sigma^2}{2 \beta}.
\end{align*}
Based on the above analyses, it is straightforward to see that
\begin{align}
\label{eq:dyn_B_t}
    \sum_{t=1}^T B_t \leq \frac{\sigma^2}{\beta} + 2\sum_{t=1}^T \eta_{t+1}\norm{\nabla s_t(y_{t-1}) - \nabla \hat{s}_t(y_{t-1})}^2_*.
\end{align}
Notice that by the definition $\eta_t$, we have
\begin{align*}
    \norm{\nabla s_t(y_{t-1}) - \nabla \hat{s}_t(y_{t-1})}^2_* = \frac{\sigma^2}{\lambda}\left( \frac{1}{\eta_{t+1}} - \frac{1}{\eta_t} \right),
\end{align*}
and as a result, 
\begin{align*}
    \sum_{t=1}^T B_t 
    \leq \frac{\sigma^2}{\beta} + \frac{2\sigma^2}{\lambda}\sum_{t=1}^T \eta_{t+1}\left( \frac{1}{\eta_{t+1}} - \frac{1}{\eta_t} \right).
\end{align*}
Using Lemma \ref{lemma:log_bound} to upper bound the RHS of the inequality above, we have that 
\begin{align*}
    \sum_{t=1}^T B_t &
    \leq \frac{\sigma^2}{\beta} + \frac{2\sigma^2}{\lambda}\sum_{t=1}^T \log \left( \frac{\eta_{t+1}^{-1}}{\eta_t^{-1}} \right) = \frac{\sigma^2}{\beta} + \frac{2\sigma^2}{\lambda}  \left( \log\left( \frac{1}{\eta_{T+1}} \right) - \log\left( \frac{1}{\eta_1} \right) \right) \\
    & = \frac{\sigma^2}{\beta} + \frac{2\sigma^2}{\lambda} \log\left( \frac{\eta_{1}}{\eta_{T+1}} \right),
\end{align*}
which immediately yields
\begin{align}
\label{eq:Th3-Bt}
    \sum_{t=1}^T B_t 
\leq \frac{\sigma^2}{\beta} + \frac{2\sigma^2}{\lambda} \log\left( 1 + \frac{\lambda}{2\beta \sigma^2}D'_T \right).
\end{align}

\textbf{(Regret upper bound)} In light of \eqref{eq:Th3-At} and \eqref{eq:Th3-Bt}, we arrive at
\begin{align*}
\Regret_T \leq 2\beta R^2 + \frac{\sigma^2}{\beta} + \frac{2\sigma^2}{\lambda} \log\left( 1 + \frac{\lambda}{2\beta \sigma^2}D'_T \right).
\end{align*}
The lemma immediately follows.
\end{proof}

Finally, for the euclidean case (i.e. $\B_h(x,y) = \frac{1}{2}\norm{x-y}^2_2$) and choosing $\lambda = \alpha/2$, we have
\begin{align*}
\lambda \B_h(x^*,y_t) - \frac{1}{\eta_t}\B_h(y_t,x_t) - \alpha\B_h(x^*,x_t) 
&= \frac{\alpha}{4}\norm{x^*-y_t}^2_2 - \frac{1}{2\eta_t}\norm{y_t-x_t}^2_2 - \frac{\alpha}{2}\norm{x^*-x_t} \\
&\leq \frac{\alpha}{2}\norm{x_t-y_t}^2_2 - \frac{1}{2\eta_t}\norm{y_t-x_t}^2_2 \leq 0,
\end{align*}
where the second inequality follows from $\norm{a-b}^2 \leq 2\norm{a-c}^2 + 2\norm{c-b}^2$ and the third inequality follows from $\eta_t^{-1} \geq \beta \geq \alpha$. Thus, we have shown that Assumption \ref{ass:bregman_reg} holds for all $t$, and Theorem \ref{theorem:OptCMD_strong} follows from Lemma \ref{lemma:OptCMD_strong_general}. $\Box$

\subsubsection{Proof of Theorem \ref{theorem:OptDCMD_convex}}

Let $u_t \in \X$. Following similar steps to the ones from the proof of Theorem \ref{theorem:OptCMD_convex}, one can show that
\begin{align*}
    f_t(x_t) - f_t(u_t) \leq A_t + B_t + C_t,
\end{align*}
where $A_t$, $B_t$ and $C_t$ are defined as
\begin{align*}
    A_t &:= \|\nabla s_t(x_t) - \nabla \hat{s}_t(y_{t-1})\|_*\|x_t - \tilde{y}_t\| - \frac{1}{2\eta_t}\B_h(\tilde{y}_t,x_t) - \frac{1}{\eta_t}\B_h(x_t,y_{t-1}),
\end{align*}
$$
B_t := \frac{1}{\eta_t}\big( \B_h(u_t,y_{t-1}) - \B_h(u_t,\tilde{y}_t) \big) \quad \text{and} \quad C_t := \Delta_t- \frac{1}{2\eta_t}\B_h(\tilde{y}_t,x_t).
$$
Moreover, still following steps similar to the proof of Theorem \ref{theorem:OptCMD_convex}, we can show that
\begin{equation}
    \label{eq:A_bound_dyna}
    \sum_{t=1}^T A_t \leq \left(4 + \frac{R^2}{2} \right)\left(V_T' + \sqrt{4\beta^2 + D_T'}\right).
\end{equation}

\textbf{(Upper bounding $\sum_{t=1}^T B_t$)}

Adding $\pm \frac{1}{\eta_t} \B_h(u_{t+1},y_t)$ and $\pm\frac{1}{\eta_t}\B_h(\Phi_t(u_t),y_t)$ to $B_t$ and summing the result over $t=1,\ldots,T$, we get
\begin{align*}
    \sum_{t=1}^T B_t 
    = \sum_{t=1}^T \frac{1}{\eta_t}\Big(&\B_h(u_t,y_{t-1}) - \B_h(u_{t+1},y_t) + \B_h(u_{t+1},y_t) \\
    &\quad - \B_h(\Phi_t(u_t),y_t) + \B_h(\Phi_t(u_t),\Phi_t(\tilde{y}_t)) - \B_h(u_t,\tilde{y}_t)\Big),
\end{align*}
where we made use of $y_t=\Phi_t(\tilde{y}_t)$. By Assumption \ref{ass:lipschtz_breg}, it holds that for some positive real $\gamma$ 
\begin{align*}
    \B_h(u_{t+1},y_t) - \B_h(\Phi_t(u_t),y_t) \leq \gamma \norm{u_{t+1}-\Phi_t(u_t)}.
\end{align*}
By Assumption \ref{ass:nonexpansiveness_Phi}, it further holds that 
\begin{align*}
    \B_h(\Phi_t(u_t),\Phi_t(\tilde{y}_t)) - \B_h(u_t,\tilde{y}_t) \leq 0.
\end{align*}
By virtue of the last two inequalities, we arrive at
\begin{align}
    \label{eq:Thm1_At_initial}
    \sum_{t=1}^T B_t \leq \sum_{t=1}^T \frac{1}{\eta_t}\big(&\B_h(u_t,y_{t-1}) - \B_h(u_{t+1},y_t) + \gamma\norm{u_{t+1}-\Phi_t(u_t)}\big).
\end{align}
Next, observe that
\begin{align*}
    \sum_{t=1}^T \frac{1}{\eta_t}\big(\B_h(u_t,y_{t-1})  - \B_h(u_{t+1},y_t)\big)
    &\leq \frac{1}{\eta_1} \B_h(u_1,y_0)  + \sum_{t=2}^T \left(\frac{1}{\eta_t} - \frac{1}{\eta_{t-1}} \right)\B_h(u_{t},y_{t-1}) \\
    &\leq \frac{R^2}{\eta_1} + R^2\sum_{t=2}^T \left(\frac{1}{\eta_t} - \frac{1}{\eta_{t-1}} \right) \leq \frac{R^2}{\eta_T} ,   
\end{align*}
where we made use Assumption \ref{ass:regularity}. Considering inequality \eqref{eq:Thm1_At_initial}, one can conclude based on the above arguments that
\begin{align}
\label{eq:B_bound_dyna}
\sum_{t=1}^T B_t
&\leq \frac{R^2}{\eta_T} + \sum_{t=1}^T \left(\frac{\gamma}{\eta_t} \norm{u_{t+1}-\Phi_t(u_t)}\right) \leq \frac{R^2}{\eta_T} + \frac{\gamma}{\eta_T}\sum_{t=1}^T \norm{u_{t+1}-\Phi_t(u_t)} =\frac{1}{\eta_T} (R^2+\gamma C'_T) ,    
\end{align}
where the second inequality follows from $\eta_{t}\geq \eta_{t+1}$. Putting \eqref{eq:A_bound_dyna} and \eqref{eq:B_bound_dyna} together, we arrive at
\begin{align}
    \Regret_T 
    &\leq \sum_{t=1}^T \left( A_t + B_t + C_t \right) \nonumber \\
    &\leq \left(4 + \frac{R^2}{2} \right)\left(V_T' + \sqrt{4\beta^2 + D_T'}\right) + (R^2+\gamma C'_T)\frac{1}{\eta_T} + \sum_{t=1}^T \left(\Delta_t- \frac{1}{2\eta_t}\B_h(y_t,x_t)\right) \nonumber \\
    &\leq \left( 4 + \frac{3R^2}{2} + \gamma C'_T \right)\left(V_T'+\sqrt{4\beta^2 + D_T'}\right) + \sum_{t=1}^T \left(\Delta_t - \frac{1}{2\eta_t}\B_h(y_t,x_t)\right) \label{eq:bound_12_dyna} \\
    &\leq \left( 5 + \frac{3R^2}{2} + \gamma C'_T \right)\left(V_T'+\sqrt{4\beta^2 + D_T'}\right) \label{eq:bound_1_dyna},
\end{align}
where we used the definition of $\eta_{T+1}$ and the fact that $V_T' = \sum_{t=1}^T |\Delta_t|$. Define $c := 5 + \frac{3R^2}{4} +\gamma C'_T$ and $\lambda_T := \sum_{t=1}^T \left( |\Delta_t| -  \frac{1}{4c\eta_t}\|x_t - y_t\|^2\right)$. By following the same steps of the last part of the proof of Theorem \ref{theorem:OptCMD_convex}, one can show that 
\begin{align}
    \label{eq:bound_2_dyna}
    \Regret_T 
    &\leq c\sqrt{2R^2+2c}\sqrt{\sum_{t=1}^T G_t^2} + \left( 4 + \frac{3R^2}{4}+\gamma C'_T \right)\sqrt{4\beta^2 + D_T'}.
\end{align}
Finally, combining \eqref{eq:bound_1_dyna} and \eqref{eq:bound_2_dyna}, arrive at
\begin{align*}
    \Regret_T 
    &\leq \min\left\{ \left( 5 + \frac{3R^2}{4} + \gamma C'_T \right)V_T', \  c\sqrt{2R^2+2c}\sqrt{\sum_{t=1}^T G_t^2} \right\}  + \left( 4 + \frac{3R^2}{4} + \gamma C'_T \right)\sqrt{4\beta^2 + D_T'} \\
    &= O\left(\left(1 + C'_T \right)\left(1 + \sqrt{D_T'} + \min\left\{ V_T', \sqrt{(1 + C'_T) T} \right\}\right)\right).
\end{align*}
This concludes the proof. $\Box$

\subsection{Proof of Theorem \ref{theorem:OptDCMD_implicit}}

Similarly to the beginning of the proof of Theorem \ref{theorem:OptDCMD_convex}, one can show that
\begin{align*}
    f_t(x_t) - f_t(u_t) 
    &\leq B_t + C_t,
\end{align*}
where $B_t$ and $C_t$ are defined as
$$
B_t := \frac{1}{\eta_t}\big( \B_h(u_t,y_{t-1}) - \B_h(u_t,\tilde{y}_t) \big) \quad \text{and} \quad C_t := \Delta_t- \frac{1}{\eta_t}\B_h(\tilde{y}_t,x_t).
$$
Next, continuing following the proof of Theorem \ref{theorem:OptDCMD_convex}, we have that
$$
\sum_{t=1}^T B_t
\leq \frac{R^2}{\eta_T} + \sum_{t=1}^T \frac{\gamma}{\eta_t} \norm{u_{t+1}-\Phi_t(u_t)} \leq \frac{1}{\eta_T}\left( R^2 + \gamma \tau \right).
$$
Thus, we have that
\begin{align}
    \Regret_T 
    &\leq \sum_{t=1}^T \left(B_t + C_t \right) \leq \frac{1}{\eta_T}\left( R^2 + \gamma \tau \right) + \sum_{t=1}^T \left(\Delta_t- \frac{1}{\eta_t}\B_h(y_t,x_t)\right) \nonumber \\
    &\leq \left(\frac{R^2}{\tau} + \gamma +1 \right)V_T' - \sum_{t=1}^T\frac{1}{\eta_t}\B_h(y_t,x_t). \nonumber
\end{align}
Again following the steps of  the proof of Theorem \ref{theorem:OptDCMD_convex}, we can alternatively bound the regret by
\begin{align*}
    \Regret_T 
    &\leq c\sqrt{2R^2+2\tau c}\sqrt{\sum_{t=1}^T G_t^2} \leq O(\sqrt{(1 + \tau)T}),
\end{align*}
where $c = \frac{R^2}{\tau} + \gamma + 1$. Combining the two regret bounds, we have
$$
\Regret_T \leq O \left( \min \left\{ V_T', \ \sqrt{(1 + \tau)T} \right\} \right).
$$
This concludes the proof. $\Box$

\subsubsection{Proof of Theorem \ref{theorem:OptDMD_convex_2}}

We start the proof by following similar steps to the ones taken in the proof of Theorem \ref{theorem:OptCMD_strong}. By doing so, we arrive at
\begin{equation} \label{eq:eq1}
\DRegret_T = \sum_{t=1}^T (f_t(x_t) - f_t(u_t)) \leq \sum_{t=1}^T A_t + \sum_{t=1}^T B_t,    
\end{equation}
where
\begin{align*}
    A_t := \|\nabla s_t(x_t) - \nabla \hat{s}_t(y_{t-1})\|_*\|x_t - \tilde{y}_t\| - \frac{1}{\eta_t}\B_h(\tilde{y}_t,x_t) - \frac{1}{\eta_t}\B_h(x_t,y_{t-1}),
\end{align*}
and
\begin{align*}
    B_t & := \frac{1}{\eta_t}(\B_h(u_t,y_{t-1}) - \B_h(u_t,\tilde{y}_t)). 
\end{align*}

\textbf{(Upper bounding $\sum_{t=1}^T A_t$)} We proceed by bounding $\sum_{t=1}^T A_t$ in the sequel. Recall that by definition, $\eta_t \leq 1/(2\beta)$ for all $t$. Thus, by following similar steps as taken in the proof of Theorem \ref{theorem:OptCMD_convex}, we get
$$
\sum_{t=1}^T A_t \leq \frac{R^2}{2\eta_{T+1}} +  2\sum_{t=1}^T \eta_{t+1}\|\nabla s_t(y_{t-1}) - \nabla \hat{s}_t(y_{t-1})\|_*^2.
$$
Recall the definition of $\eta_t$. Invoking lemma Lemma \ref{lemma:sqrt_bound_2}, we arrive at
\begin{equation}
\label{eq:eq3}
\sum_{t=1}^T A_t \leq \frac{R^2}{2\eta_{T+1}} + 4\sqrt{(\theta_T + D'_T)(1 + C'_T)}.  
\end{equation}

\textbf{(Upper bounding $\sum_{t=1}^T B_t$)} Following similar steps as taken in the proof of Theorem \ref{theorem:OptDCMD_convex}, we can bound
$$
\sum_{t=1}^T B_t
\leq \frac{R^2}{\eta_T} + \sum_{t=1}^T \frac{\gamma}{\eta_t} \norm{u_{t+1}-\Phi_t(u_t)}.
$$
Next, notice that
\begin{align*}
\sum_{t=1}^T \frac{\gamma}{\eta_t}\norm{u_{t+1}-\Phi_t(u_t)} 
&= \gamma\sum_{t=1}^T \left( \frac{1}{\eta_t} - \frac{1}{\eta_{t+1}} + \frac{1}{\eta_{t+1}} \right)\norm{u_{t+1}-\Phi_t(u_t)} \\
&= \gamma\sum_{t=1}^T \left( \frac{1}{\eta_t} - \frac{1}{\eta_{t+1}}  \right)\norm{u_{t+1}-\Phi_t(u_t)} + \gamma\sum_{t=1}^T \frac{1}{\eta_{t+1}}\norm{u_{t+1}-\Phi_t(u_t)} \\
&\leq \frac{\gamma R^2}{\eta_1} + \gamma\sum_{t=1}^T \frac{1}{\eta_{t+1}}\norm{u_{t+1}-\Phi_t(u_t)},
\end{align*}
where for the last inequality, we assumed without loss of generality that $\norm{u_{t+1}-\Phi_t(u_t)} \leq R^2$. Following these same steps again, and using the fact that $\eta_1 = \eta_2$, we get
$$
\sum_{t=1}^T \frac{\gamma}{\eta_t}\norm{u_{t+1}-\Phi_t(u_t)} \leq \frac{2\gamma R^2}{\eta_1} + \gamma\sum_{t=1}^T \frac{1}{\eta_{t+2}}\norm{u_{t+1}-\Phi_t(u_t)}.
$$
Recall the definition of $\eta_t$ in Theorem \ref{theorem:OptDMD_convex_2}. Invoking Lemma \ref{lemma:sqrt_bound_2}, we get
$$
\gamma\sum_{t=1}^T \frac{1}{\eta_{t+2}}\norm{u_{t+1}-\Phi_t(u_t)} \leq 2\gamma\sqrt{(\theta_{T+2} + D'_{T+1})(1 + C'_T)}
$$
Back to our upper bound on $A_t$, we now have
\begin{equation} \label{eq:eq2}
\sum_{t=1}^T B_t
\leq \frac{2\gamma R^2}{\eta_1} + \frac{R^2}{\eta_T} + 2\gamma\sqrt{(\theta_{T+2} + D'_{T+1})(1 + C'_T)}.
\end{equation}

\textbf{(Regret upper bound)} Considering equations \eqref{eq:eq1}, \eqref{eq:eq2} and \eqref{eq:eq3}, it holds that
\begin{align*}
\DRegret_T &\leq \frac{2\sigma^2 R^2}{\eta_1} + \frac{3R^2}{2\eta_T} + (4+2\gamma)\sqrt{(\theta_{T+2} + D'_{T+1})(1 + C'_T)}.   
\end{align*}
This concludes the proof. $\Box$

\section{Numerical Experiments}
\label{sec:numerical}

\subsection{Tracking Dynamical Parameters}\label{sec:tracking}

In this section, we employ a strategy based on Algorithm \ref{alg:OptDCMD} in a parameter tracking problem. The scenario presented in this section is based on the numerical experiment of \cite{shahrampour2017distributed}. Denote the parameters to be tracked by $u_t \in \Real^4$.
These parameters have dynamics described by the linear model $u_{t+1} = A u_t + v_t$.
Similarly to \cite{shahrampour2017distributed}, we emphasize that our online learning results hold even when the noise is adversarial with an unknown structure.
For this experiment, we use
$$
A = \begin{bmatrix}
1 & 0.1 & 0 & 0 \\
0 & 1 & 0.1 & 0 \\
0 & 0 & 1 & 0.1 \\
0 & 0 & 0 & 1 \\
\end{bmatrix} 
\quad \text{and} \quad
v_t = 
\begin{cases} 
2 & \text{if } \tilde{v}_t > 0 \\
-1 & \text{if } \tilde{v}_t \leq 0
\end{cases}
$$
where $\tilde{v}_t$ is Gaussian noise with a random covariance matrix, and the inequalities in the definition of $v_t$ are component-wise. The cost at time $t$ is defined as $f_t(x_t) = \frac{1}{2}\norm{x_t - u_t}^2_2 + \|x_t\|_1$, where $s_t(x_t) := \frac{1}{2}\norm{x_t - u_t}^2_2$, $r_t(x_t) := \|x_t\|_1$ and $x_t$ is the output of our tracking algorithm. We assume the Player has access to $\Phi_t(x) = Ax$, which is an approximate model of the dynamics of $u_t$.

To choose its action sequence $\{x_t\}_{t=1}^T$, the Player employs a variation of Algorithm \ref{alg:OptDCMD} with $h(x) = \frac{1}{2}\|x\|^2_2$ (i.e., the euclidean setup), with the difference that in the update rule of $\tilde{y}_t$, we use a constant step-size $\eta_t = 1$. This change was inspired by \cite{mokhtari2016online}, and the fact that $\frac{1}{2}\norm{x_t - u_t}^2_2$ is smooth and strongly convex. For the update rule of $x_t$, we use the step-size defined in Theorem \ref{theorem:OptDCMD_convex} (notice that since the nonsmooth component of the cost $f_t$ is fixed, $V'_t = 0$ for all $t$). We consider the following gradient prediction models:
\begin{enumerate}
    \item \textbf{perfect}: a perfect model $\nabla \hat{s}_t(y_{t-1}) := \nabla s_t(y_{t-1})$;
    \item \textbf{noisy}: a noisy model $\nabla \hat{s}_t(y_{t-1}) := \nabla s_t(y_{t-1}) + w_t$;
    \item \textbf{noisy+bias}: a noisy prediction model plus a bias term $\nabla \hat{s}_t(y_{t-1}) := \nabla s_t(y_{t-1}) + w_t - 1$;
    \item \textbf{previous}: a prediction model that uses the previous cost gradient $\nabla \hat{s}_t(y_{t-1}) := \nabla s_{t-1}(y_{t-1})$;
    \item \textbf{random}: a random prediction model $\nabla \hat{s}_t(y_{t-1}) := w_t$,
\end{enumerate}
where $w_t \sim \mathcal{N}(0, 0.5I)$. 
As a benchmark, we use the Dynamic Mirror Descent (DMD) algorithm of Hall and Willett \cite{hall2013dynamical} with a constant step-size $\eta = 1$ and a dynamic version of Algorithm \ref{alg:OptMD}, which also uses the dynamical model $\Phi_t$ to update the $y_t$ variable. We refer to this algorithm as dynamic \ref{alg:OptMD}.

Denote the regrets of Algorithm \ref{alg:OptDCMD}, the DMD algorithm and the dynamic \ref{alg:OptMD} by $\DRegret_t$(OptDCMD), $\DRegret_t$(DMD) and $\DRegret_t$(d-OptMD), respectively. The experiments are repeated 100 times, and for each experiment, a new trajectory $\{u_t\}_{t=1}^T$ was generated. The shaded areas correspond to one standard deviation for Figure  \ref{fig:regret_track_DMD} and $0.1$ times one standard deviation for Figure \ref{fig:regret_track_dOptMD}. Figure \ref{fig:regret_track_DMD} depicts the difference $\DRegret_t$(OptDCMD)-$\DRegret_t$(DMD). One can observe that all the models that use some kind of information about future gradients (\textbf{perfect}, \textbf{noisy}, \textbf{noisy+bias}) were able to perform better than the benchmark. This shows that indeed Algorithm \ref{alg:OptDCMD} was able to exploit predictive information about the problem. Moreover, model \textbf{previous} and \textbf{random} also perform better than the benchmark on average, showing the robustness of our algorithm against inaccurate gradient predictions. Figure \ref{fig:regret_track_DMD} depicts the difference $\DRegret_t$(OptDCMD)-$\DRegret_t$(d-OptMD). As can be seen, Algorithm \ref{alg:OptDCMD} performs better than the benchmark for all predictions models, illustrating the advantage of the composite updates Algorithm \ref{alg:OptDCMD} compared with Algorithm \ref{alg:OptMD}.

\begin{figure*}
\centering
\subfloat[\ref{alg:OptDCMD} versus DMD.]{\label{fig:regret_track_DMD}\includegraphics[scale=0.55]{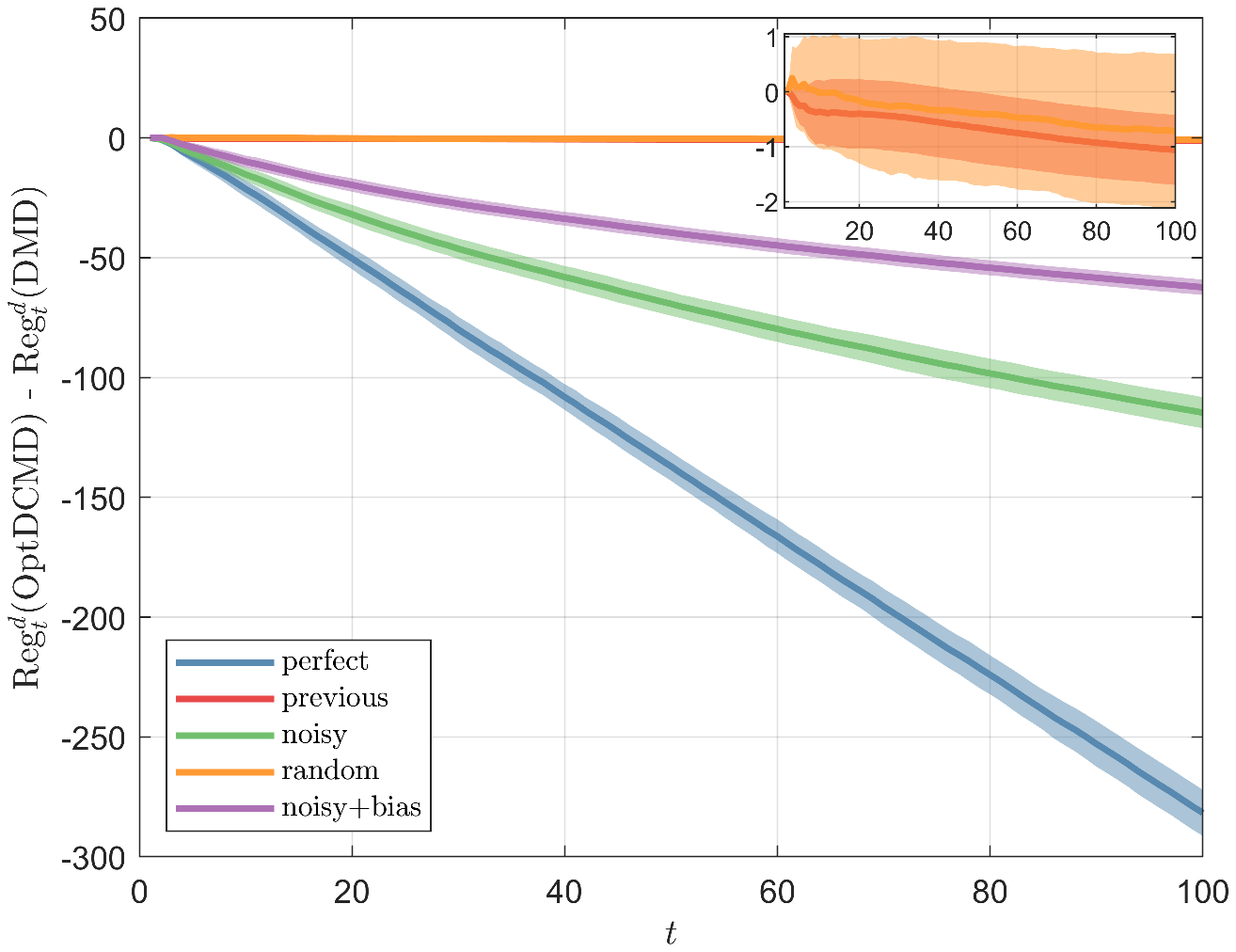}}
\subfloat[\ref{alg:OptDCMD} versus dynamic \ref{alg:OptMD}.]{\label{fig:regret_track_dOptMD}\includegraphics[scale=0.55]{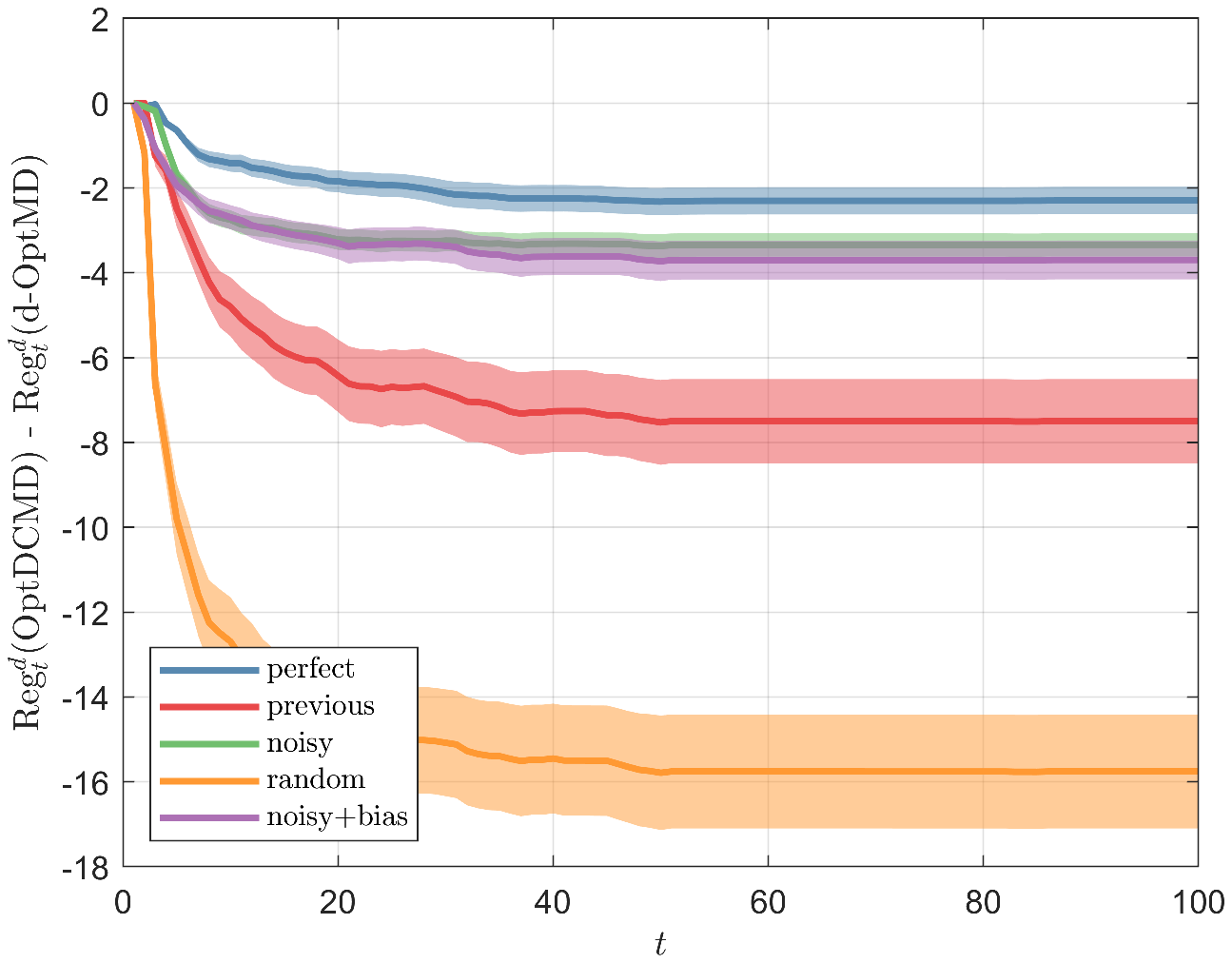}} 
\caption{Regret difference between the DMD algorithm, the dynamic version of Algorithm \ref{alg:OptMD} and Algorithm \ref{alg:OptDCMD}, for different gradient prediction models.}
\label{fig:regret_track}
\end{figure*}

\subsection{Portfolio Selection}
\label{sec:portfolio}

In this section, we apply the result of Theorem \ref{theorem:OptCMD_convex} in a portfolio selection problem. Suppose that an investor (or the Player) has $n$ assets in a Market (or Nature). Let the Player's action $x$ be a probability distribution over $n$ assets. The action set $\X$ is thus $\Delta_n := \{ x \in \Real^n: x(i) \geq 0,  \sum_{i=1}^n x(i) = 1 \}$. Let the return of an asset at round $t$ be the ratio of the value of the asset between rounds $t$ and $t+1$. At round $t$, Nature chooses a strictly positive return vector $r_t \in \Real^n_{>0}$ such that each entry of $r_t$ corresponds to the return of an asset. The Player's wealth ratio between rounds $t$ and $t+1$ is $\inner{r_t}{x_t}$. Let the Player's gain at round $t$ be $\log(\inner{r_t}{x_t})$. In a game of $T$ rounds, the goal of the Player is to maximize $\sum_{t=1}^T \log(\inner{r_t}{x_t})$ or, equivalently, to minimize $\sum_{t=1}^T -\log(\inner{r_t}{x_t})$. Hence, we have $f_t(x)=-\log(\inner{r_t}{x})$ and $\nabla f_t(x) = -r_t/\inner{r_t}{x}$, for all $x\in\X$ \footnote{See \cite{hazan2016introduction} for a more detailed description of this problem.}. Notice that in this scenario, there is no nonsmooth component in the cost $f_t$, and Algorithm \ref{alg:OptCMD} reduces to Algorithm \ref{alg:OptMD}.

We assume that the Player has prediction models of the return vector $r_t$, denoted by $\hat{r}_t$.
Thus, in light of the approaches proposed in this paper, we define
\begin{equation}\label{eq:Mt}
    \nabla \hat{f}_t(y_{t-1}) := -\frac{\hat{r}_t}{\inner{\hat{r}_t}{ y_{t-1}}}.
\end{equation}
In what follows, we show how the Player can employ Algorithm \ref{alg:OptMD} to decide its action sequence $\{x_t\}_{t=1}^T$ considering the static regret \eqref{eq:regret}. Since the costs are convex, the Player uses the step-size rule of Theorem \ref{theorem:OptCMD_convex} in Algorithm \ref{alg:OptMD} (with $V_t = 0$). We assume the return of each asset at each time $t$ is bounded as $r_{\text{min}} \leq r_t \leq r_{\text{max}}$ (component-wise). By assuming $r_{\text{min}} = 0.5$ and $r_{\text{max}} = 1.5$, we can set the smoothness parameter $\beta = 9$. Since $\Delta_n$ is the $n$-dimensional simplex, we let $h(x)$ be the \emph{negative entropy} function $\sum^n_{i=1}x(i)\log(x(i))$. Observe that $h$ is $1$-strongly convex w.r.t. $\norm{\cdot}_1$ \cite{bubeck2015convex}. We consider the following prediction models for the returns vector:
\begin{enumerate}
    \item \textbf{MA(k)}: a Moving Average prediction model  model $\tilde{r}_t := \frac{1}{k} \sum_{i=1}^k r_{t-i}$;
    \item \textbf{previous}: a model that uses the previous return vector as its prediction $\tilde{r}_t := r_{t-1}$;
    \item \textbf{noisy}: a noisy, unbiased predictor model of the true returns vector $\tilde{r}_t := r_t + v_t$, where $v_t \sim \mathcal{N}(0,0.3)$;
    \item \textbf{random}: a random predictor, where the entries of $\tilde{r}_t$ are chosen uniformly between $r_{\text{min}}$ and $r_{\text{max}}$.
    \item \textbf{recursiveLS(k)}: for each stock, we have a prediction model of the form $\tilde{r}_t = w_1r_{t-1} + w_2r_{t-2} + \dots + w_kr_{t-k} + w_{k+1}$, where the weights $w_1, \dots, w_{k+1}$ are updated online, using a recursive least squares algorithm.
\end{enumerate}

However, instead of using the output of these models directly into Equation \eqref{eq:Mt}, we will use $\hat{r}_t = g(\tilde{r}_t)$. The function $g(r)$ is defined as
$$
g(r) := 
\begin{cases} 
r_{\text{max}} & \text{if } r > 1 \\
1 & \text{if } r = 1 \\
r_{\text{min}} & \text{if } r < 1 
\end{cases}
$$
and is applied component-wise for vector inputs. The interpretation behind passing the predictions $\tilde{r}_t$ through $g$ is that, instead of using the exact predictions given by our models, we use $\tilde{r}_t$ only as an indication if a given stock is predicted to increase or decrease its value in the next round.

To simulate a stock market, we use six real-world datasets: NYSE(O), NYSE(N), DJIA, TSE, SP500, and MSCI. A detailed description of these datasets can be found in \cite{li2015moving}. Let the number of assets of each dataset be $N$. As a benchmark of each experiment, we employ the \emph{Constant Uniform Portfolio} (CUP) strategy, that is, a Player that chooses $x_t = [1/N, \dots, 1/N]$, for all $t\in[T]$. For the datasets considered in this experiment, the CUP strategy performed better than the Algorithm \ref{alg:OMD}, for any $\eta_t > 0$ and $x_0 = [1/N, \dots, 1/N]$. 

Denote the regrets of Algorithm \ref{alg:OptMD} and CUP strategies by  $\Regret_T$(OptMD) and $\Regret_T$(CUP), respectively.
Figure \ref{fig:regret_portifolio} depicts the difference $\Regret_t$(OptMD)$-\Regret_t$(CUP) for each considered dataset. 
The experiment was repeated 10 times and the shaded areas correspond to one standard deviation. 
As expected, for all datasets, the \textbf{noisy} model achieved the best performance, since it uses information of $r_t$ in the prediction $\hat{r}_t$. More interestingly, we notice that for all datasets except DJIA, the \textbf{recursiveLS(6)} prediction model performed better than all other models.
Moreover, this model also performed better than the CUP benchmark strategy.
In other words, at time $t$, we were able to generate and exploit the predictive information about the return of each stock, using only information available up to time $t-1$.
Another interesting conclusion we can draw from Figure \ref{fig:regret_portifolio} is that, in general, using either the previous return or a simple moving average as predictions lead to poor performance for the algorithm. 
Finally, when using the \textbf{random} models (i.e., gradient predictions uncorrelated with the true gradients), Algorithm \ref{alg:OptMD} performed generally similarly to the CUP benchmark strategy.
This indicates that our approach can also be robust to bad gradient predictions (see Remark \ref{remark:OptCMD_convex}).

\begin{figure*}
\centering
\subfloat[NYSE(N).]{\label{fig:regret_diff_portifolio_dataset3}\includegraphics[scale=0.4]{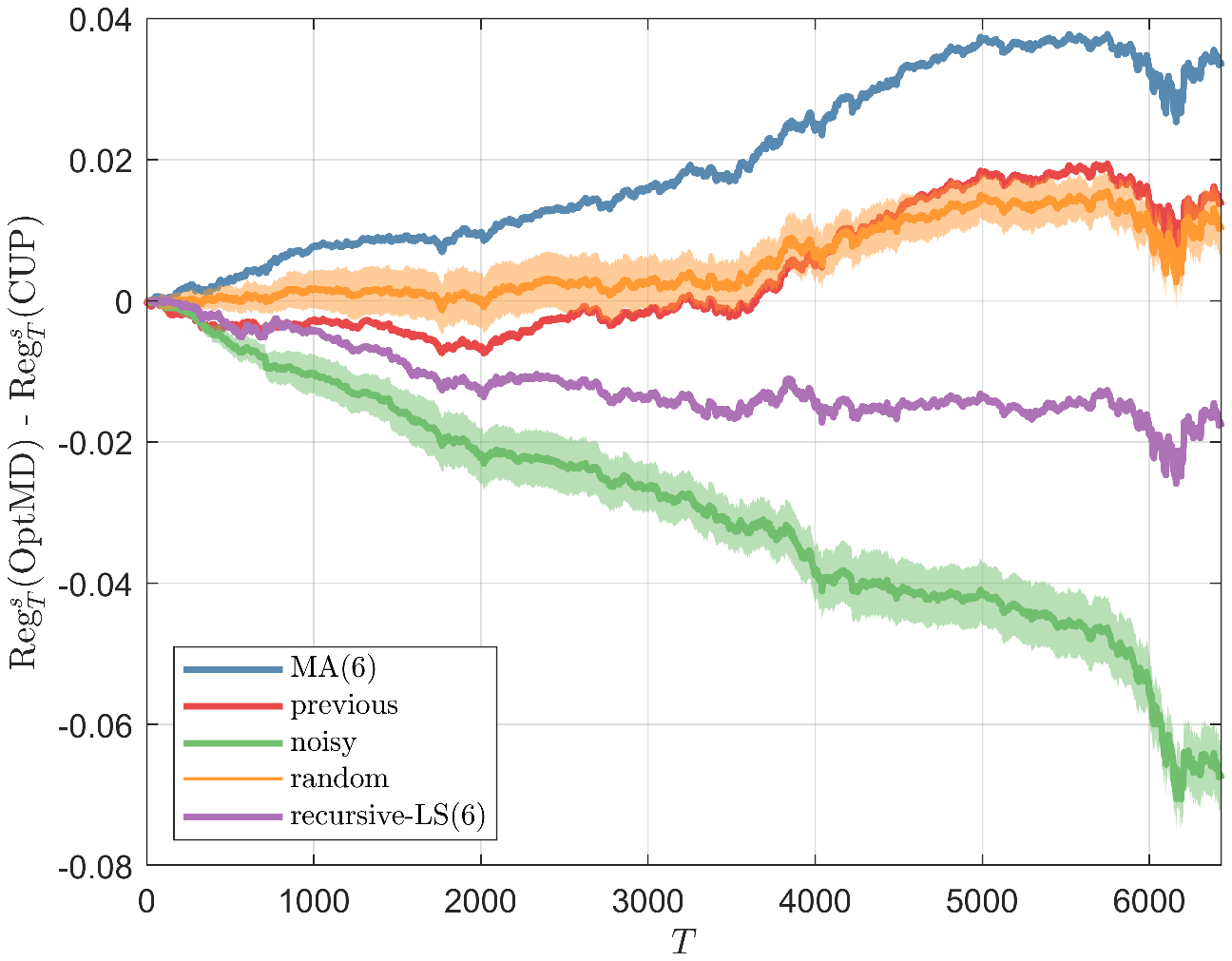}}
\subfloat[NYSE(O).]{\label{fig:regret_diff_portifolio_dataset4}\includegraphics[scale=0.4]{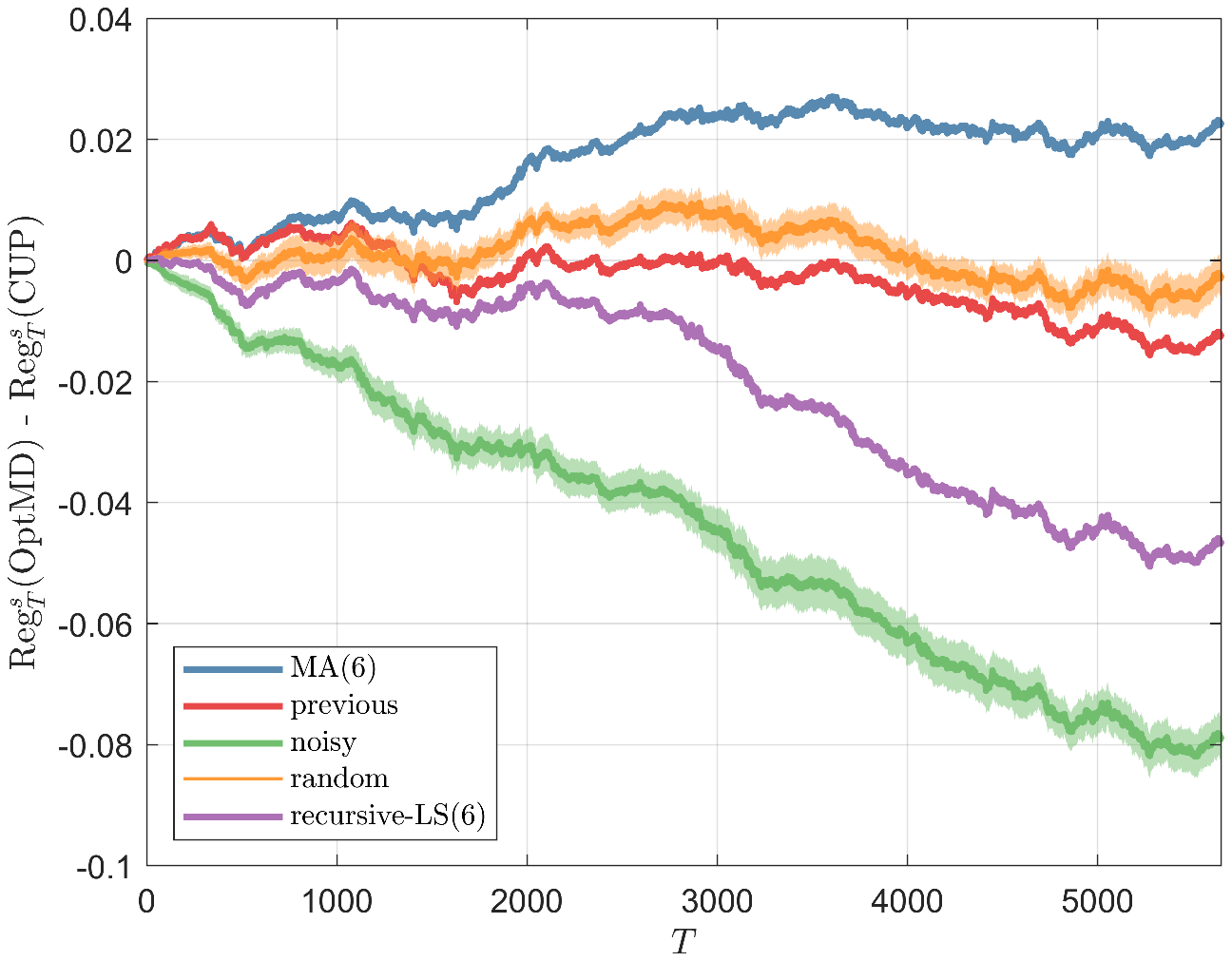}} 
\subfloat[DJIA.]{\label{fig:regret_diff_portifolio_dataset5}\includegraphics[scale=0.4]{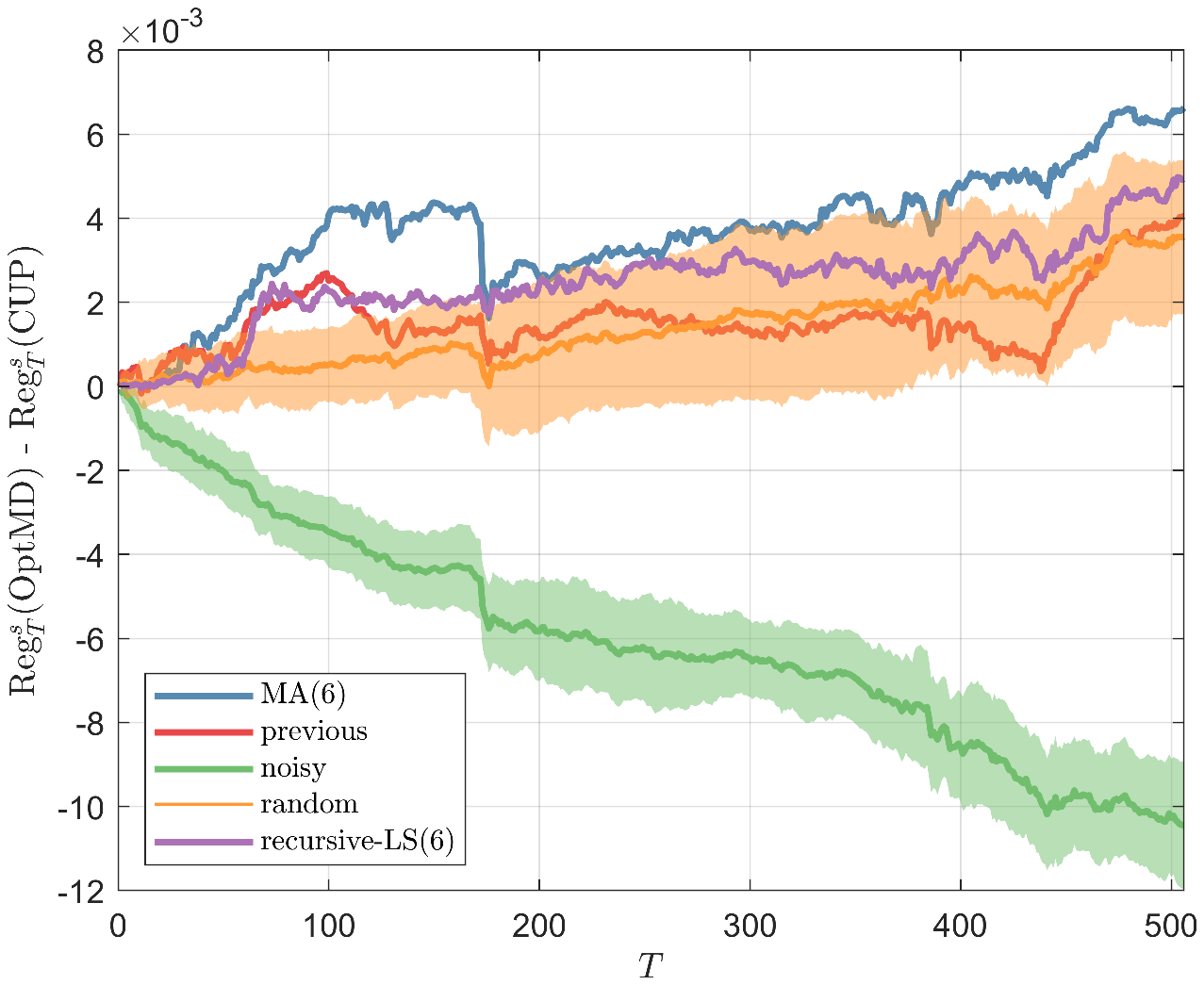}} \\
\subfloat[MSCI.]{\label{fig:regret_diff_portifolio_dataset6}\includegraphics[scale=0.4]{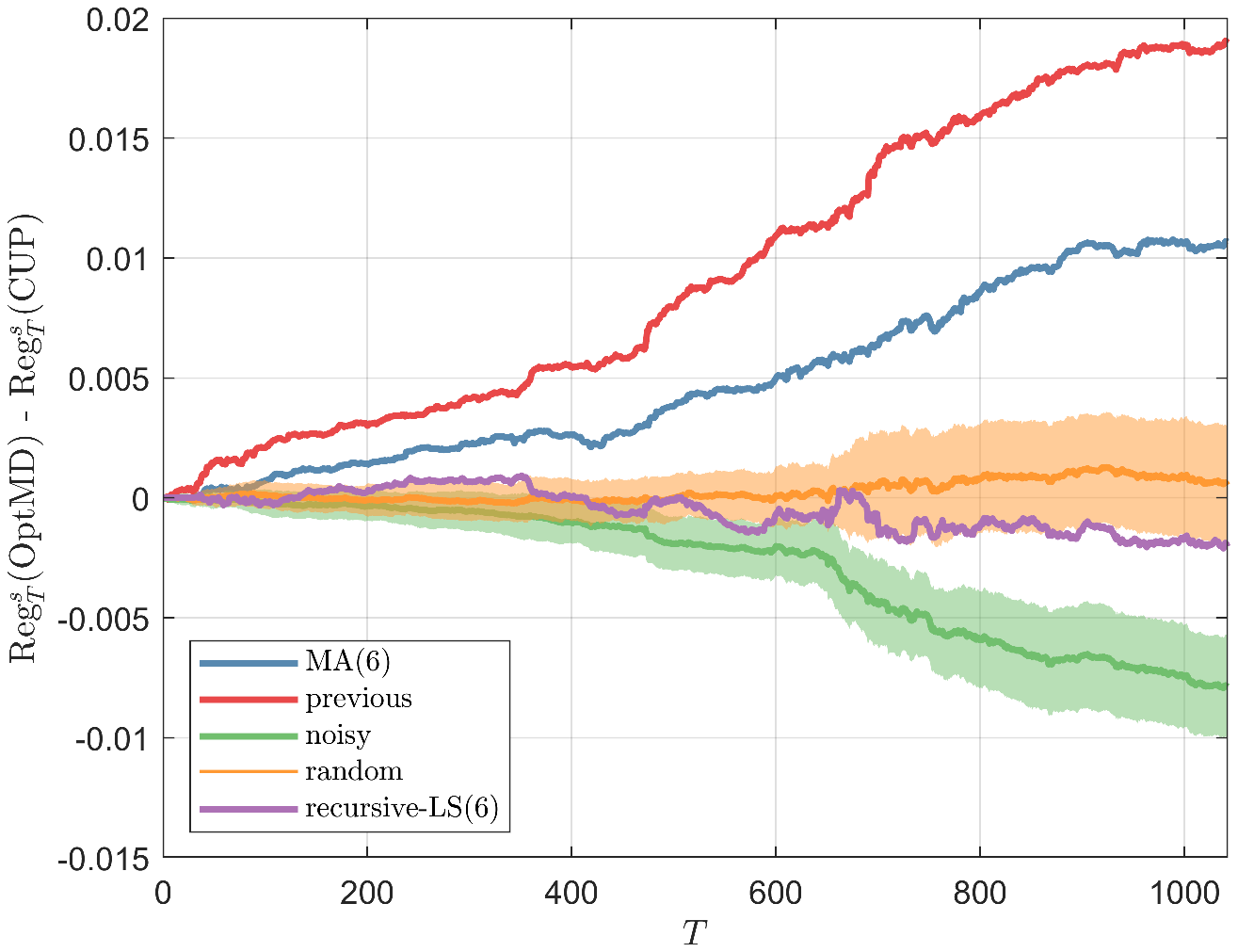}}
\subfloat[SP500.]{\label{fig:regret_diff_portifolio_dataset7}\includegraphics[scale=0.4]{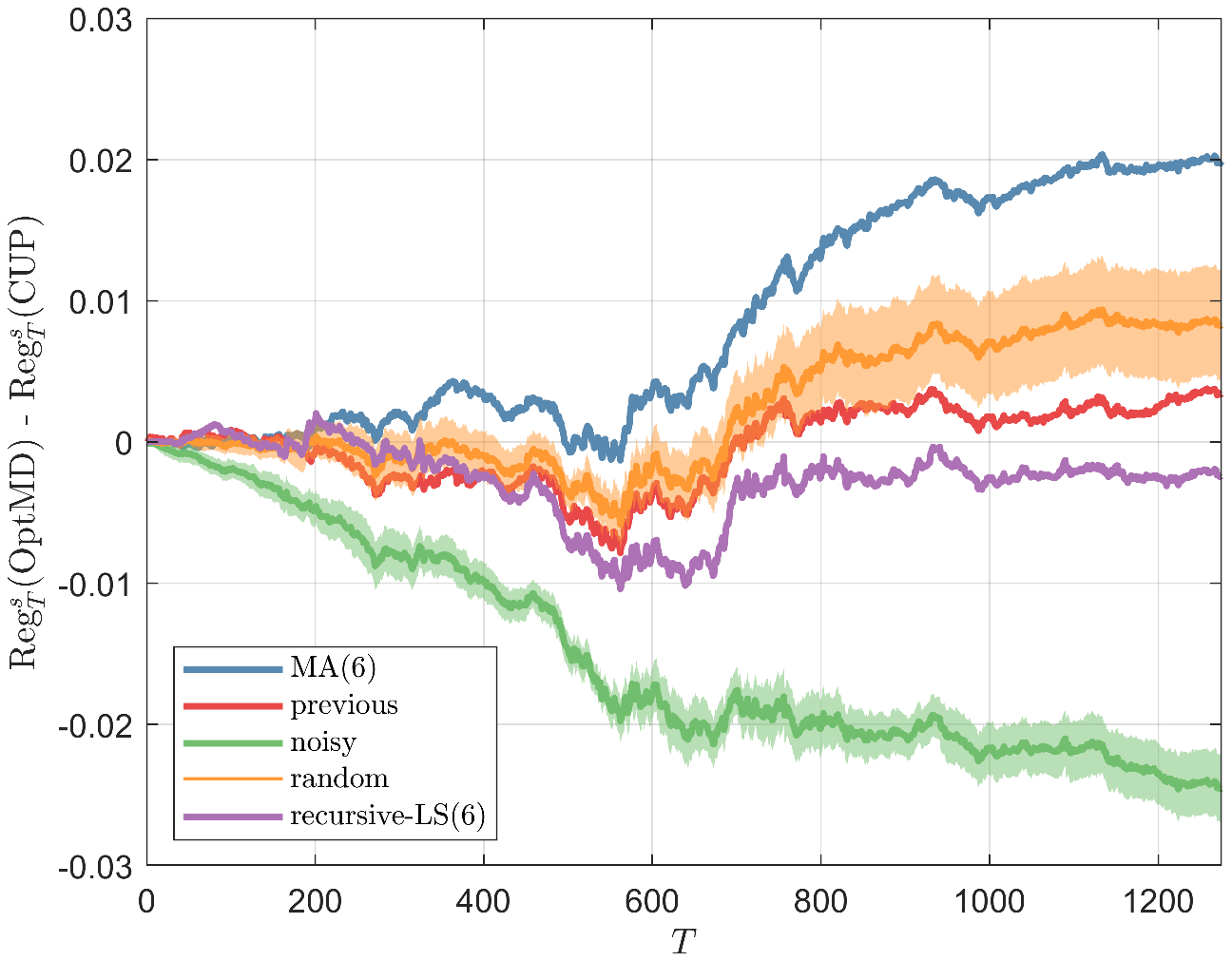}}
\subfloat[TSE.]{\label{fig:regret_diff_portifolio_dataset8}\includegraphics[scale=0.4]{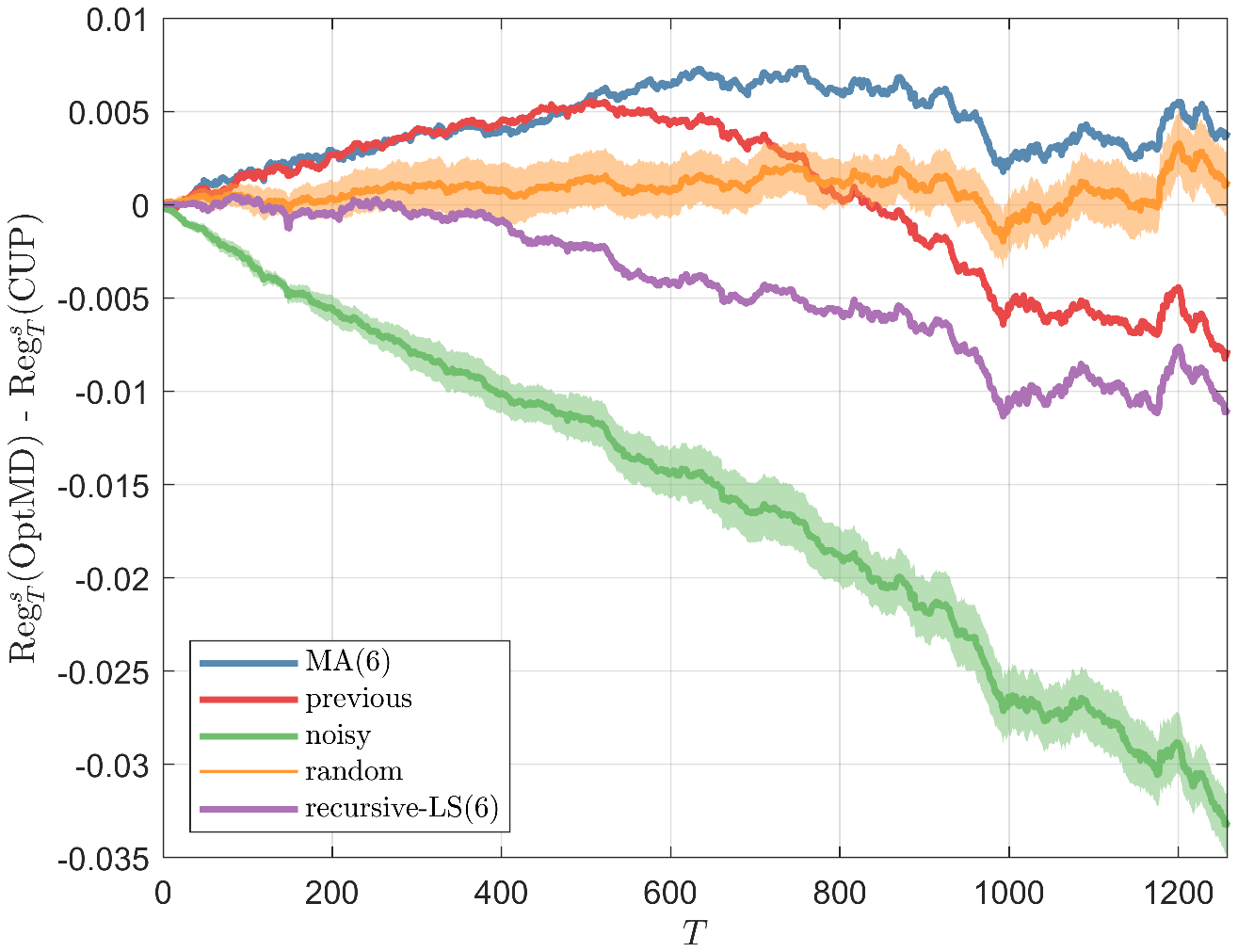}}
\caption{Algorithm \ref{alg:OptMD} applied to the Portfolio Selection problem.}
\label{fig:regret_portifolio}
\end{figure*}

\appendix

\section{Literature Landscape and Summary of Results}
\label{app:tables}

\renewcommand{\arraystretch}{1.5}

In Table \ref{table:comparison}, we present our work in the OCO literature with respect to computational (i.e., not composite costs vs. composite costs) and information (i.e., no predictions vs. with predictions) features of OCO problems. For a more detailed discussion of the literature, see Section \ref{sec:introduction}.

\newpage

{\small 
\begin{table}
\centering
\begin{tabular}{ c || c | c |} 
    \cline{2-3} 
    & \makecell{No predictions \\ $\nabla \hat{s}_t = 0, ~\hat{r}_t = 0$} & \makecell{With predictions \\ $\nabla \hat{s}_t \neq 0$ (and $ \hat{r}_t \neq 0$)} \\
    \hline\hline
    \multicolumn{1}{|l||}{\makecell{Not composite \\ $r_t = 0$}} & \cite{zinkevich2003online, hazan2007logarithmic} & \cite{rakhlin2013online, rakhlin2013optimization, ho2019exploiting, jadbabaie2015online, chang2021online} \\
    \hline
    \multicolumn{1}{|l||}{\makecell{Composite \\ $r_t \neq 0$}} & \cite{duchi2010composite, kulis2010implicit, campolongo2020temporal, campolongo2021closer} & \textbf{this work} \\
    \hline
\end{tabular}
\caption{Examples of OCO literature considering composite and prediction features.
}
\label{table:comparison}
\end{table}
}

Tables \ref{table:perfect_func_pred} and \ref{table:func_pred} summarize the contributions of this work concerning the static regret bounds, presenting a comparison for different cases of gradient predictions. In particular, Table \ref{table:perfect_func_pred} concerns the case of \textit{perfect} function predictions, whereas Table \ref{table:func_pred} concerns the case of \textit{general} function prediction. For a more detailed discussion of these results, see Remarks \ref{remark:OptCMD_convex} and \ref{remark:OptCMD_strong}.



{\small 
\begin{table}[h!]
\begin{tabular}{|cc||c|c|c|}
\hline
\multicolumn{2}{|c||}{\makecell{Perfect function prediction \\ $\hat{r}_t = r_t$}} & \makecell{General case \\ (\textbf{this work})} & Worst-case & Perfect prediction \\
\hline\hline
\multicolumn{2}{|c||}{Gradient prediction} & $\nabla \hat{s}_t$ & $\nabla \hat{s}_t = $ any & $\nabla \hat{s}_t = \nabla s_t$ \\ 
\hline
\multicolumn{1}{|c|}{\multirow{2}{*}{Convex costs}} & $\eta_t$ & $1/\sqrt{D'_{t-1} + 4\beta^2}$ & $ O(1/\sqrt{t})$ & $1/(2\beta)$ \\ \cdashline{2-5}
\multicolumn{1}{|c|}{} & Regret & \makecell{$O(1 + \sqrt{D'_T})$ \\ (Thm. \ref{theorem:OptCMD_convex})} & $O(\sqrt{T})$ \cite{zinkevich2003online} & $O(1)$ \cite{ho2019exploiting} \\ 
\hline
\multicolumn{1}{|c|}{\multirow{2}{*}{\makecell{Strongly convex \\ costs}}} & $\eta_t$ & $O(1/(D'_{t-1}+ 2\beta))$ & $ O(1/t)$ & $1/(2\beta)$ \\
\cdashline{2-5} 
\multicolumn{1}{|c|}{} & Regret & \makecell{$O(1 + \log(1+ D'_T))$ \\ (Thm. \ref{theorem:OptCMD_strong})} & $O(\log(T))$ \cite{hazan2007logarithmic} & $O(1)$ \cite{ho2019exploiting} \\
\hline
\end{tabular}
\caption{Generality of the static regret bounds in the case of \textit{perfect} function predictions (i.e., $\hat{r}_t = r_t$).}
\label{table:perfect_func_pred}
\end{table}}

{\small
\begin{table}[h!]
\centering
\begin{tabular}{| c || c | c | c |} 
    \hline
    General function prediction & \makecell{General case \\ (\textbf{this work})} & Worst-case & Perfect prediction \\
    \hline\hline
    Gradient prediction & $\nabla \hat{s}_t$ & $\nabla \hat{s}_t = $ any & $\nabla \hat{s}_t = \nabla s_t$ \\
    \hline
    $\eta_t$ & $1/\sqrt{4\beta^2 + (V'_{t-1})^2 + D'_{t-1}}$ & $ O(1/\sqrt{t})$ & $1/\sqrt{4\beta^2 + (V'_{t-1})^2}$\\
    \hdashline
    Regret & \makecell{$O\left(1 + \sqrt{D'_T} + \min\left\{ V'_T, \sqrt{T} \right\}\right)$ \\ (Thm. \ref{theorem:OptCMD_convex})} & $O(\sqrt{T})$ \cite{zinkevich2003online} & $O\left(1 + \min\left\{ V'_T, \sqrt{T} \right\}\right)$ \cite{campolongo2020temporal} \\
    \hline
\end{tabular}
\caption{Generality of the static regret bounds in the case of \textit{general} function predictions.}
\label{table:func_pred}
\end{table}}

\bibliographystyle{plain}
\bibliography{references.bib}

\begin{thebibliography}{10}

\bibitem{abernethy2008optimal}
Jacob Abernethy, Peter Bartlett, Alexander Rakhlin, and Ambuj Tewari.
\newblock Optimal strategies and minimax lower bounds for online convex games.
\newblock In {\em Conference on Learning Theory (COLT 2008)}, pages 415--423,
  2008.

\bibitem{pmlr-v97-agarwal19c}
Naman Agarwal, Brian Bullins, Elad Hazan, Sham Kakade, and Karan Singh.
\newblock Online control with adversarial disturbances.
\newblock In {\em Proceedings of the 36th International Conference on Machine
  Learning}, pages 111--119, 2019.

\bibitem{agarwal2019logarithmic}
Naman Agarwal, Elad Hazan, and Karan Singh.
\newblock Logarithmic regret for online control.
\newblock {\em Advances in Neural Information Processing Systems}, 2019.

\bibitem{beck2017first}
Amir Beck.
\newblock {\em First-order methods in optimization}.
\newblock SIAM, 2017.

\bibitem{besbes2015non}
Omar Besbes, Yonatan Gur, and Assaf Zeevi.
\newblock Non-stationary stochastic optimization.
\newblock {\em Operations research}, 63(5):1227--1244, 2015.

\bibitem{bhaskara2020online}
Aditya Bhaskara, Ashok Cutkosky, Ravi Kumar, and Manish Purohit.
\newblock Online learning with imperfect hints.
\newblock In {\em Proceedings of the 37th International Conference on Machine
  Learning (ICML 2020)}, 2020.

\bibitem{bubeck2015convex}
S.~Bubeck.
\newblock {\em Convex Optimization: Algorithms and Complexity}.
\newblock Foundations and Trends in Machine Learning. Now Publishers, 2015.

\bibitem{bubeck2011introduction}
S{\'e}bastien Bubeck.
\newblock Introduction to online optimization.
\newblock {\em Lecture notes}, 2011.

\bibitem{campolongo2020temporal}
Nicol{\`o} Campolongo and Francesco Orabona.
\newblock Temporal variability in implicit online learning.
\newblock {\em Advances in neural information processing systems}, 2020.

\bibitem{campolongo2021closer}
Nicolo Campolongo and Francesco Orabona.
\newblock A closer look at temporal variability in dynamic online learning.
\newblock {\em preprint arXiv:2102.07666}, 2021.

\bibitem{cesa2006prediction}
Nicolo Cesa-Bianchi and G{\'a}bor Lugosi.
\newblock {\em Prediction, learning, and games}.
\newblock Cambridge University Press, 2006.

\bibitem{chang2021online}
Ting-Jui Chang and Shahin Shahrampour.
\newblock On online optimization: Dynamic regret analysis of strongly convex
  and smooth problems.
\newblock In {\em Proceedings of the AAAI Conference on Artificial
  Intelligence}, 2021.

\bibitem{chen2015online}
Niangjun Chen, Anish Agarwal, Adam Wierman, Siddharth Barman, and Lachlan~LH
  Andrew.
\newblock Online convex optimization using predictions.
\newblock In {\em Proceedings of the 2015 ACM SIGMETRICS International
  Conference on Measurement and Modeling of Computer Systems}, pages 191--204,
  2015.

\bibitem{chen2016using}
Niangjun Chen, Joshua Comden, Zhenhua Liu, Anshul Gandhi, and Adam Wierman.
\newblock Using predictions in online optimization: Looking forward with an eye
  on the past.
\newblock {\em ACM SIGMETRICS Performance Evaluation Review}, 2016.

\bibitem{dekel2017online}
Ofer Dekel, Arthur Flajolet, Nika Haghtalab, and Patrick Jaillet.
\newblock Online learning with a hint.
\newblock In {\em Advances in Neural Information Processing Systems (NIPS
  2017)}, pages 5299--5308, 2017.

\bibitem{duchi2010composite}
John~C Duchi, Shai Shalev-Shwartz, Yoram Singer, and Ambuj Tewari.
\newblock Composite objective mirror descent.
\newblock In {\em COLT}, volume~10, pages 14--26. Citeseer, 2010.

\bibitem{hall2013dynamical}
Eric Hall and Rebecca Willett.
\newblock Dynamical models and tracking regret in online convex programming.
\newblock In {\em Proceedings of the 30th International Conference on Machine
  Learning (ICML 2013)}, pages 579--587, 2013.

\bibitem{hall2015online}
Eric~C Hall and Rebecca~M Willett.
\newblock Online convex optimization in dynamic environments.
\newblock {\em IEEE Journal of Selected Topics in Signal Processing},
  9(4):647--662, 2015.

\bibitem{hazan2016introduction}
Elad Hazan.
\newblock Introduction to online convex optimization.
\newblock {\em Foundations and Trends in Optimization}, 2(3-4):157--325, 2016.

\bibitem{hazan2007logarithmic}
Elad Hazan, Amit Agarwal, and Satyen Kale.
\newblock Logarithmic regret algorithms for online convex optimization.
\newblock {\em Machine Learning}, 2007.

\bibitem{pmlr-v117-hazan20a}
Elad Hazan, Sham Kakade, and Karan Singh.
\newblock The nonstochastic control problem.
\newblock In {\em Proceedings of the 31st International Conference on
  Algorithmic Learning Theory}, pages 408--421, 2020.

\bibitem{ho2019exploiting}
Nam Ho-Nguyen and Fatma K{\i}l{\i}n{\c{c}}-Karzan.
\newblock Exploiting problem structure in optimization under uncertainty via
  online convex optimization.
\newblock {\em Mathematical Programming}, 177(1-2):113--147, 2019.

\bibitem{jadbabaie2015online}
Ali Jadbabaie, Alexander Rakhlin, Shahin Shahrampour, and Karthik Sridharan.
\newblock Online optimization: Competing with dynamic comparators.
\newblock In {\em Proceedings of the 18th International Conference on
  Artificial Intelligence and Statistics (AISTATS 2015)}, pages 398--406, 2015.

\bibitem{kivinen1997exponentiated}
Jyrki Kivinen and Manfred~K Warmuth.
\newblock Exponentiated gradient versus gradient descent for linear predictors.
\newblock {\em Information and Computation}, 1997.

\bibitem{kulis2010implicit}
Brian Kulis and Peter~L Bartlett.
\newblock Implicit online learning.
\newblock In {\em Proceedings of the 27th International Conference on Machine
  Learning (ICML)}, 2010.

\bibitem{lesage2020predictive}
Antoine Lesage-Landry, Iman Shames, and Joshua~A Taylor.
\newblock Predictive online convex optimization.
\newblock {\em Automatica}, 113:108771, 2020.

\bibitem{li2015moving}
Bin Li, Steven~CH Hoi, Doyen Sahoo, and Zhi-Yong Liu.
\newblock Moving average reversion strategy for on-line portfolio selection.
\newblock {\em Artificial Intelligence}, 2015.

\bibitem{li2022robustness}
Tongxin Li, Ruixiao Yang, Guannan Qu, Guanya Shi, Chenkai Yu, Adam Wierman, and
  Steven Low.
\newblock Robustness and consistency in linear quadratic control with untrusted
  predictions.
\newblock {\em Proceedings of the ACM on Measurement and Analysis of Computing
  Systems}, 6(1):1--35, 2022.

\bibitem{li2019online}
Yingying Li, Xin Chen, and Na~Li.
\newblock Online optimal control with linear dynamics and predictions:
  Algorithms and regret analysis.
\newblock {\em Advances in Neural Information Processing Systems}, 32, 2019.

\bibitem{li2020leveraging}
Yingying Li and Na~Li.
\newblock Leveraging predictions in smoothed online convex optimization via
  gradient-based algorithms.
\newblock {\em Advances in Neural Information Processing Systems},
  33:14520--14531, 2020.

\bibitem{li2020online}
Yingying Li, Guannan Qu, and Na~Li.
\newblock Online optimization with predictions and switching costs: Fast
  algorithms and the fundamental limit.
\newblock {\em IEEE Transactions on Automatic Control}, 2020.

\bibitem{lin2020online}
Yiheng Lin, Gautam Goel, and Adam Wierman.
\newblock Online optimization with predictions and non-convex losses.
\newblock {\em Proc. ACM Meas. Anal. Comput. Syst.}, 2020.

\bibitem{mcmahan2010unified}
H~Brendan McMahan.
\newblock A unified view of regularized dual averaging and mirror descent with
  implicit updates.
\newblock {\em preprint arXiv:1009.3240}, 2010.

\bibitem{mohri2016accelerating}
Mehryar Mohri and Scott Yang.
\newblock Accelerating online convex optimization via adaptive prediction.
\newblock In {\em AISTATS}, 2016.

\bibitem{mokhtari2016online}
Aryan Mokhtari, Shahin Shahrampour, Ali Jadbabaie, and Alejandro Ribeiro.
\newblock Online optimization in dynamic environments: Improved regret rates
  for strongly convex problems.
\newblock In {\em 55th IEEE Conference on Decision and Control (CDC 2016)},
  pages 7195--7201, 2016.

\bibitem{nagahara2015maximum}
Masaaki Nagahara, Daniel~E Quevedo, and Dragan Ne{\v{s}}i{\'c}.
\newblock Maximum hands-off control: a paradigm of control effort minimization.
\newblock {\em IEEE Transactions on Automatic Control}, 61(3):735--747, 2015.

\bibitem{nagahara2013sparse}
Masaaki Nagahara, Daniel~E Quevedo, and Jan {\O}stergaard.
\newblock Sparse packetized predictive control for networked control over
  erasure channels.
\newblock {\em IEEE Transactions on Automatic Control}, 59(7):1899--1905, 2013.

\bibitem{nesterov2004introductory}
Yurii Nesterov.
\newblock {\em Introductory lectures on convex optimization: A basic course}.
\newblock Springer, 2004.

\bibitem{nesterov2018lectures}
Yurii Nesterov.
\newblock {\em Lectures on convex optimization}, volume 137.
\newblock Springer, 2018.

\bibitem{parikh2014proximal}
Neal Parikh and Stephen Boyd.
\newblock Proximal algorithms.
\newblock {\em Foundations and Trends in optimization}, 1(3):127--239, 2014.

\bibitem{rakhlin2013online}
Alexander Rakhlin and Karthik Sridharan.
\newblock Online learning with predictable sequences.
\newblock In {\em Conference on Learning Theory (COLT)}, 2013.

\bibitem{rakhlin2013optimization}
Sasha Rakhlin and Karthik Sridharan.
\newblock Optimization, learning, and games with predictable sequences.
\newblock In {\em Advances in Neural Information Processing Systems (NIPS
  2013)}, pages 3066--3074, 2013.

\bibitem{ravier2019prediction}
R.~J. {Ravier}, A.~R. {Calderbank}, and V.~{Tarokh}.
\newblock Prediction in online convex optimization for parametrizable objective
  functions.
\newblock In {\em 58th IEEE Conference on Decision and Control}, pages
  2455--2460, 2019.

\bibitem{shahrampour2017distributed}
Shahin Shahrampour and Ali Jadbabaie.
\newblock Distributed online optimization in dynamic environments using mirror
  descent.
\newblock {\em IEEE Transactions on Automatic Control}, 63(3):714--725, 2017.

\bibitem{shalev2012online}
Shai Shalev-Shwartz.
\newblock Online learning and online convex optimization.
\newblock {\em Foundations and Trends in Machine Learning}, 2012.

\bibitem{shalev2007logarithmic}
Shai Shalev-Shwartz and Yoram Singer.
\newblock Logarithmic regret algorithms for strongly convex repeated games.
\newblock {\em The Hebrew University}, 2007.

\bibitem{song2018fully}
Chaobing Song, Ji~Liu, Han Liu, Yong Jiang, and Tong Zhang.
\newblock Fully implicit online learning.
\newblock {\em preprint arXiv:1809.09350}, 2018.

\bibitem{wagener2019online}
Nolan Wagener, Ching-An Cheng, Jacob Sacks, and Byron Boots.
\newblock An online learning approach to model predictive control.
\newblock {\em Proceedings of Robotics: Science and Systems (RSS)}, 2019.

\bibitem{yang2014regret}
Tianbao Yang, Mehrdad Mahdavi, Rong Jin, and Shenghuo Zhu.
\newblock Regret bounded by gradual variation for online convex optimization.
\newblock {\em Machine learning}, 95(2):183--223, 2014.

\bibitem{yu2020power}
Chenkai Yu, Guanya Shi, Soon-Jo Chung, Yisong Yue, and Adam Wierman.
\newblock The power of predictions in online control.
\newblock {\em Advances in Neural Information Processing Systems},
  33:1994--2004, 2020.

\bibitem{zhang2018adaptive}
Lijun Zhang, Shiyin Lu, and Zhi-Hua Zhou.
\newblock Adaptive online learning in dynamic environments.
\newblock In {\em Advances in Neural Information Processing Systems (NIPS
  2018)}, pages 1323--1333, 2018.

\bibitem{zinkevich2003online}
Martin Zinkevich.
\newblock Online convex programming and generalized infinitesimal gradient
  ascent.
\newblock In {\em Proceedings of the 20th International Conference on Machine
  Learning (ICML 2003)}, pages 928--936, 2003.

\end{thebibliography}

\end{document}